\documentclass{scrartcl}

\usepackage{amsmath,amssymb,amsfonts,amsthm}
\usepackage{mathtools}
\usepackage{nicefrac} 
\usepackage{enumitem} 
\usepackage[T1]{fontenc} 

\usepackage{algorithm}
\usepackage{algpseudocode}

\usepackage{tikz-cd} 

\usepackage{subcaption}
\usepackage{booktabs}
\usepackage{siunitx} 
\usepackage{colortbl}
\usepackage[dvipsnames]{xcolor}

\usepackage{comment} 
\usepackage{url}
\usepackage{hyperref}
\usepackage{doi}
\newtheorem{theorem}{Theorem}
\newtheorem{lemma}[theorem]{Lemma} 
\newtheorem{proposition}[theorem]{Proposition}
\newtheorem{remark}[theorem]{Remark}

\newtheorem{corollary}[theorem]{Corollary}
\newtheorem{example}[theorem]{Example}

\newcommand{\N}{\mathbb{N}}
\newcommand{\R}{\mathbb{R}}

\renewcommand{\d}{\mathrm{d}}
\renewcommand{\i}{\mathrm{i}}
\newcommand{\e}{\mathrm{e}}

\newcommand{\tT}{\mathrm{T}}

\DeclareMathOperator{\sinc}{sinc}
\DeclareMathOperator*{\argmin}{arg\,min}

\title{Numerical Methods for Kernel Slicing}
\author{Nicolaj Rux\footnote{Chemnitz University of Technology, DE} \and Johannes Hertrich\footnote{Univesité Paris Dauphine-PSL and Inria Mokaplan, Paris, FR} \and Sebastian Neumayer\footnotemark[1]}

\begin{document}
\maketitle
\begin{abstract}\noindent
Kernels are key in machine learning for modeling interactions.
Unfortunately, brute-force computation of the related kernel sums scales quadratically with the number of samples.
Recent Fourier-slicing methods lead to an improved linear complexity, provided that the kernel can be sliced and its Fourier coefficients are known.
To obtain these coefficients, we view the slicing relation as an inverse problem and present two algorithms for their recovery. 
Extensive numerical experiments demonstrate the speed and accuracy of our methods.
\end{abstract}

\section{Introduction}
\label{sec:Intro}
Kernel methods arise in various machine learning applications, for example, in density estimation \cite{P1962,R1956}, classification with support vector machines \cite{S2011,SC2008}, principal component analysis for dimensionality reduction \cite{SS2002,SC2004}, and maximum mean discrepancies (MMD) as statistical distance measure \cite{GBRSS2006,S2002}.
The latter are also of interest in gradient flow dynamics \cite{AKSG2019,HHABCS2023,SteNeuRux2025} as well as Stein variational gradient descent \cite{LQ2016}.
Moreover, dot product kernels appear in the self attention mechanism of transformers \cite{CACP2025,PPYS2021,VSPUJGKP2017}.
In all these applications, the computational bottleneck is usually the evaluation of expressions of the form
\begin{align}\label{eq:Fw}
    s_m \coloneq  \sum_{n=1}^N F(\|x_n-y_m\|)w_n, \quad  m=1,\ldots,M,
\end{align}
where $F\in C([0,\infty))$ is a radial (kernel) basis function, $x_1,\ldots, x_N,y_1,\ldots,y_M\in \R^d$ are (potentially high dimensional) samples, and $w\in \R^N$ is a weight.
The computation of \eqref{eq:Fw}  as matrix-vector product of $[F(\|x_n-y_m\|)]_{n,m=1}^{N,M}$ with the weights $w$ requires $\mathcal O(NMd)$ operations.
This is impractical for large datasets.

Reducing the complexity of the kernel summations \eqref{eq:Fw} has a rich history.
Classical approaches are fast Fourier summations \cite{PSN2004} and fast multipole methods \cite{GR1987}, see Section~\ref{sec:related_work} for details.
While these reduce the complexity in $M$ and $N$ to $\mathcal O(M+N)$, they rely on fast Fourier transforms or a segmentation of the space.
Hence, we have an exponential dependence on the dimension $d$, which makes them infeasible for $d>4$.

Instead, we consider the slicing algorithm for fast kernel summations \cite{H2024}.
Its basic idea is to find  $f\in  L^1_\mathrm{loc}([0,\infty))$ such that $F$ can be represented as
\begin{equation}\label{eq:slicing}
    F(\|x\|)=\frac{1}{\omega_{d-1}}\int_{\mathbb S^{d-1}} f( |\langle \xi,x\rangle|)\d \xi, \quad x\in \R^d,
\end{equation}
where $\omega_{d-1}= \nicefrac{2\pi^{\nicefrac{d}{2}}}{\Gamma(\nicefrac{d}{2})}$ is the surface measure of the $d$-dimensional sphere $\mathbb S^{d-1}$.
Then, the computation of the kernel sums \eqref{eq:Fw} can be reduced to the one-dimensional case by discretizing the integral in \eqref{eq:slicing}.
Using (random) directions $\xi_1,\ldots, \xi_P\in \mathbb S^{d-1}$, we obtain
\begin{equation}\label{eq:SlicingApprox}
s_m = \int_{\mathbb S^{d-1}} \sum_{n=1}^N f(|\langle x_n - y_m,\xi\rangle|)w_n \d \xi \approx \frac1P \sum_{p=1}^P \sum_{n=1}^N f(|\langle x_n,\xi_p\rangle-\langle y_m,\xi_p\rangle|)w_n.
\end{equation}
Now, the right hand side of \eqref{eq:SlicingApprox} consists out of $P$ one-dimensional kernel summations, which can be evaluated efficiently using fast Fourier summations.
Bounds on the discretization error for the integral in \eqref{eq:slicing} and quasi Monte Carlo rules for choosing the directions $\xi_p$ are given in \cite{HJQ2025}.

As remaining challenge, we have to find a function $f$ such that \eqref{eq:slicing} is fulfilled.
This $f$ depends both on the basis function $F$ and the dimension $d$ of the samples $x_n$, $y_m$.
By using non-constructive arguments, \cite{RQS2025} proved existence and uniqueness of a $f$ fulfilling \eqref{eq:slicing} if $F$ is sufficiently regular.
For specific $F$, the associated $f$ can be obtained via power series representations \cite{H2024} or Fourier transforms \cite{RQS2025}, see Section~\ref{sec:slicing_formulas} for an overview. 
However, these approaches require the analytic computation of high-dimensional Fourier transforms or the inversion of integral operators, which are both challenging in practice.

\paragraph{Contribution and Outline}
First, we revisit the necessary background in Section~\ref{sec:backgrounds}.
Then, we formulate the recovery of $f$ as an inverse problem with forward operator~\eqref{eq:slicing}.
Building on the theory established in Section~\ref{sec:slicing_formulas}, we propose two algorithms for recovering $f$ in a fixed dimension~$d$.
More precisely, these algorithms, detailed in Section~\ref{sec:optim}, compute the coefficients of the cosine series of $f(|\cdot|)$ required for the fast Fourier summation method.
Moreover, we prove error estimates for both proposed algorithms in Section~\ref{sec:error_bounds}.
Finally, Section~\ref{sec:numerics} presents numerical experiments demonstrating the efficiency and accuracy of our approaches with conclusions drawn in Section~\ref{sec:concl}.

\section{Background}\label{sec:backgrounds}

Here, we provide background for the slicing algorithm.
This includes a brief literature overview in Section~\ref{sec:related_work} and a description of the fast Fourier summation in Section~\ref{sec:fastsum_1d}.

\subsection{Related Literature}\label{sec:related_work}

\paragraph{Slicing}
The idea of projecting data to random one-dimensional subspaces first appeared in context of the Wasserstein distance \cite{RPDB2011}.
The authors of \cite{KNSS2022} generalized this procedure to kernel metrics like the MMD and observed that the sliced MMD is again a MMD with respect to a different kernel.
In \cite{HWAH2023}, the authors proved that for the Riesz kernel $K(x,y)=\|x-y\|^r$ the sliced MMD coincides with the MMD (also known as energy distance  \cite{S2002} in this setting) and used this fact to construct a fast algorithm for its computation.
Based on \eqref{eq:SlicingApprox}, \cite{H2024} introduced the slicing algorithm for fast kernel summations.
This algorithm was further investigated in \cite{HJQ2025, RQS2025}.

If $F(\|\cdot\|)$ is positive definite, slicing is closely related to random Fourier features \cite{RR2007} as outlined in \cite{RQS2025}.
However, slicing is more general, since it also applies to cases where 
$F(\|\cdot\|)$ is not positive definite or where the Fourier transform of $F$ is not a function.
Beyond the Riesz kernel, notable examples in this broader setting include the thin-plate spline and the logarithmic kernel.

Finally, note that the relation \eqref{eq:slicing} has multiple names in harmonic analysis and fractional calculus.
In it is also known as the adjoint Radon transformation \cite{R2003}, the generalized Riemann-Liouville fractional integral, which is a special case of the Erdelyi-Kober integral \cite{L1971, SKM1993}, and it appears within the Funk-Hecke formula \cite[p.~30]{M1998}.

\paragraph{Fast Kernel Summations}
In dimension $d\leq4$, several fast kernel summation methods were proposed.
These include summations based on (non-)equispaced fast Fourier transforms \cite{GRB2022,PSN2004}, fast multipole methods \cite{BN1992,GR1987}, tree-based methods \cite{MXB2015} or H- and mosaic-skeleton matrices \cite{H1999,T1996}.
For the Gauss kernel, the fast Gauss transform was proposed by \cite{GS1991} and improved by \cite{YDD2004}.
More general fast kernel transforms were considered by \cite{RAGD2022}.
In moderate dimensions (up to $d\approx 100$), brute-force GPU implementations like KeOps \cite{CFGCD2011} can still be competitive.

\subsection{Fast Fourier Summation in 1D}\label{sec:fastsum_1d}

In the following, we revisit the fast Fourier summation \cite{PSN2004} on $\R$, which is needed to compute the right hand side of \eqref{eq:SlicingApprox}.
More precisely, given samples $x_1,...,x_N\in\R$, $ y_1,...,y_M\in\R$ and weights $w_1,...,w_N\in\R$, we want to compute the sums
\begin{equation}\label{eq:s1D}
t_m \coloneqq \sum_{n=1}^N f(|x_n-y_m|)w_n, \quad  m=1\ldots,M.
\end{equation}
To this end, we expand $f(|\cdot|)$ into a truncated Fourier series 
\begin{equation}\label{eq:num_fourer_series}
    f(|x|)\approx \sum_{k=-K}^K c_k[f] \exp(\pi \i kx/T), \quad x\in [-T,T],
\end{equation}
where $T>0$ is chosen such that  $\|x_n-y_m\|\le T$.
Inserting this into \eqref{eq:s1D}, we obtain
\begin{equation}\label{eq:ffs_outer}
t_m
\approx\sum_{n=1}^N \sum_{k=-K}^K c_k[f] \exp(\pi \i k {(x_n-y_m)}/{T})w_n = \sum_{k=-K}^K c_k[f] \exp(-\pi \i k
y_m/T) \hat w_k
\end{equation}
with 
\begin{equation}\label{eq:ffs_inner}
\hat w_k\coloneqq \sum_{n=1}^N \exp(\pi \i x_n^p/T)w_n, \quad k=-K,\ldots, K.
\end{equation}
Since the sums \eqref{eq:ffs_outer} are Fourier transforms at non-equispaced data points, we can use the non-equispaced fast Fourier transform (NFFT) \cite{B1995,DR1993,PST2001}, leading to the complexity $\mathcal O(N+M+K\log K)$.
Hence, the total complexity of computing the kernel sums \eqref{eq:Fw} based on \eqref{eq:SlicingApprox} is $\mathcal O(P(N+M+K\log K))$.

\subsection{Notations}
For measurable $w\colon X\to \R$, let $L^p(X,w)$, $p\in [1,\infty]$, be the Banach space of equivalence classes of real-valued functions on $X$ with finite norm $\|f\|_{L^p(X,w)} = (\int_X |f(x)|^p w(x) \d x)^{\nicefrac{1}{p}}$.
Then, $L^2(X,w)$ is a Hilbert space with inner product $\langle f, g\rangle_{L^2(X)} = \int_X f(x)g(x)w(x)\d x$.
For brevity, we write $L^p(X)=L^p(X,w)$ if $w\equiv1$ and $L^p([0,1],\nicefrac{1}{s})=L^p([0,1],w)$ if $w(s)=\nicefrac{1}{s}$.
Moreover, recall that $L^p_\mathrm{loc}([0,\infty))$ consists of all $f\colon \R\to \R$ that are $p$-integrable on every compact $K\subseteq [0,\infty)$.
Let $f^{(n)}$ denote the $n$th weak derivative of $f \colon [0, 1] \to \mathbb{R}$.
Then, the Sobolev space
\begin{equation}
    H^n([0,1]) = \{ f \colon [0, 1] \to \mathbb{R} : f^{(k)} \in L^2([0,1]) \text{ for } k = 0, \dots, n \}
\end{equation}
is a Hilbert space with inner product
$\langle f, g \rangle_{H^n([0,1])} = \sum_{k=0}^n \langle f^{(k)}, g^{(k)}\rangle_{L^2([0,1])}$.
For $d\ge1$ and $f\in L^1(\R^d)$, the Fourier transform $\mathcal F_d\colon L^1(\R^d)\to \mathcal C_0(\R)$ is defined as
\begin{equation}\label{eq:def_fourier}
  \mathcal F_d[f](\omega) =  \int_{\R^d} f(x)\exp(-2\pi \i \langle x,\omega\rangle )\d x,\quad \omega\in \R^d.
\end{equation}

\section{The Slicing Operator}\label{sec:slicing_formulas}
Throughout, we assume $d\ge 3$ to avoid case distinctions. 
As outlined in Section \ref{sec:Intro}, we aim to find for given $F\colon[0,\infty)\to\R$ some $f\colon[0,\infty)\to\R$ such that \eqref{eq:slicing} is fulfilled.
If $f\in L^1_\mathrm{loc}([0,\infty))$, it was proven in \cite{H2024} that a pair $(F,f)$ fulfills \eqref{eq:slicing} if and only if $F=\mathcal S_d[f]$, where $\mathcal S_d[f]$ is the generalized Riemann-Liouville fractional integral of $f$ given by
\begin{equation}\label{eq:RLFI}
 \mathcal S_d[f](s)= \int_0^1f(ts) \varrho_d(t)\d t, \quad \text{where } \varrho_d(t)\coloneqq c_d(1-t^2)^{\frac{d-3}{2}}\text{ and } c_d\coloneqq \tfrac{2\Gamma(\frac{d}{2})}{\sqrt{\pi}\Gamma(\frac{d-1}{2})}.
\end{equation}
The constant $c_d$ normalizes $\varrho_d$ to a probability density on $(0,1)$.
We illustrate $\varrho_d$ in Figure~\ref{fig:varrho_d} and plot the constant $c_d$ for different dimensions in Figure~\ref{fig:cd_1e4}.
The integral equation \eqref{eq:RLFI} is a special case of Erdelyi--Kober integrals, see \cite{L1971,SKM1993}, and was previously used for computing the Radon transform of radial functions, see \cite{Maass1991,R2003}.

\begin{figure}[tbp]
\centering

  \begin{subfigure}[t]{0.48\textwidth}
    \centering
  \includegraphics[width=\textwidth]{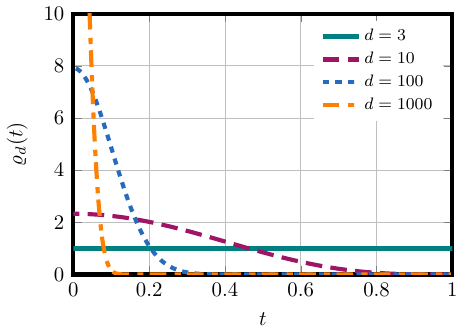}
    \caption{Density $\varrho_d(t)=c_d(1-t^2)^{\nicefrac{(d-3)}{2}}$, for $t\in [0,1]$ and $d\in \{3,10,100,\num{1e3}\}$.}
    \label{fig:varrho_d}  
  \end{subfigure}
  \hfill
  \begin{subfigure}[t]{0.48\textwidth}
    \centering
\includegraphics[width=0.96\textwidth]{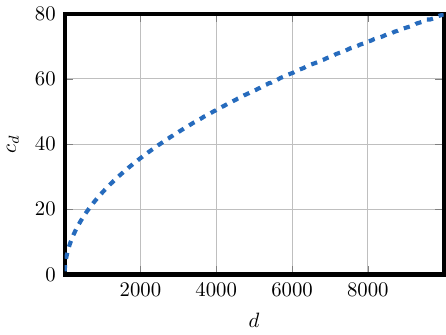}
    \caption{Normalization constant $c_d$ for dimensions $d=3\ldots, 10^4$.}
    \label{fig:cd_1e4}
  \end{subfigure}
  \caption{Densities $\varrho_d$ and normalization constant $c_d$ appearing for $\mathcal S_d$ in \eqref{eq:RLFI}.}
\end{figure}

\begin{remark}[Relevant Domain]\label{rem:01}
The value of $\mathcal{S}_d[f](s)$ is only affected by values of $f$ on $[0,s]$. 
In particular, $\mathcal S_d$ remains well-defined if we consider it as an operator $\mathcal S_d \colon \mathcal G \to \mathcal H$ for suitable function spaces $\mathcal G$ and $\mathcal H$ on $[0,1]$ instead of $[0,\infty)$.
\end{remark}

In practice, restricting our attention to the domain $[0,1]$ corresponds to the assumption that the samples $x_n$ and $y_m$ in \eqref{eq:Fw} are normalized in the sense that $\|x_n-y_m\|\leq 1$ for all $n,m$.
In this case, we also have that $\langle \xi,x_n-y_m\rangle\leq 1$ for all $\xi\in\mathbb S^{d-1}$ such that the functions $F$ and $f$ in \eqref{eq:Fw} and \eqref{eq:slicing} are only evaluated for values in $[0,1]$.

\begin{remark}\label{rem:dilations}
If the samples are not normalized, we can simply rescale them, see \cite{H2024} for details.
This relies on the property that $\mathcal S_d$ commutes with the dilation operator $\mathcal D_\alpha f(x)=f(\alpha x)$, namely $\mathcal S_d[\mathcal D_\alpha f] = D_\alpha \mathcal S_d[f]$.
\end{remark}

Next, we study the properties of \emph{the slicing operator} $\mathcal S_d$.
In particular, we are interested in the range of $\mathcal S_d$ and consider methods to invert $\mathcal S_d$ in Section~\ref{subsec:inversion}.
Afterwards, in Section~\ref{sec:operator_norm}, we study the operator norm of $\mathcal S_d \colon \mathcal G \to \mathcal H$.
The latter depends heavily on the choice of $\mathcal G$ and $\mathcal H$.

\subsection{Explicit Inversion}\label{subsec:inversion}

So far, most works focused on analytically finding $f\in\mathcal S_d^{-1}[F]$.
In the following, we summarize these approaches and extend some of the results.
It was proven in \cite[eq.~(8)]{L1971} that $\mathcal S_d$ is injective on $L_1^\mathrm{loc}([0,\infty))$.
Thus, at most one $f\in\mathcal S_d^{-1}[F]$ exists. 

\paragraph{Inversion Method 1: Eigendecomposition for Power Series} The monomials $f(t)=t^k$, $k\in(-1,\infty)$, are eigenfunctions of $\mathcal S_d$ in the sense that $\mathcal S_d[f]=\gamma_{k,d} f$ for $\gamma_{k,d}=\nicefrac{\sqrt{\pi}\Gamma(\frac{k+d}{2})}{\Gamma(\frac{d}{2})\Gamma(\frac{k+1}{2})}$. Based on this observation, we obtain that $\mathcal S_d$ is bijective as an operator $\mathcal S_d\colon \mathcal P\to \mathcal P$, where $\mathcal P\coloneqq \{g(t)=\sum_{k=0}^\infty a_kt^k,\,t\geq0\}$ is the space of globally convergent power series on $[0,\infty)$. In this case, we can compute $f=\mathcal S_d^{-1}[F]$ as detailed in \cite[Thm.\ 2.3]{H2024}.
Although many common basis functions $F$ admit a power series representation, this inversion approach typically yields a power series for $f$ with large coefficients, making it impractical to evaluate numerically.

\paragraph{Inversion Method 2: Fourier Transforms of Radial Basis Functions}
If $F \circ \|\cdot\|\in L^1([0,\infty), s^{d-1})$ and $\hat F \coloneqq \mathcal F_d^{-1}[F\circ\|\cdot\|]\in L^1(\R^d)$, the preimage $f$ can be computed via Fourier transforms as described in \cite{RQS2025}.
Recall that $\hat F$ remains radial as the Fourier transform of a radial function.
Consequently, $\hat F(x)=\hat G(\|x\|)$ for some $G\colon[0,\infty)\to\R$, which we extend to negative values by $G(-r)=G(r)$ for $r>0$.
Then, it can be proven that $F=\mathcal S_d[f]$ for $f=\frac{\omega_{d-1}}{2}\mathcal F_1[|\cdot|^{d-1}G(\cdot)]$ with $\omega_{d-1}= \nicefrac{2\pi^{\nicefrac{d}{2}}}{\Gamma(\nicefrac{d}{2})}$.
In practice, computing $\hat F$ and verifying $\hat F\in L^1(\R^d)$ is challenging.
Still, if these conditions are violated, similar results hold in a distributional sense \cite{RQS2025}.
As an example of this approach, we compute $f=\mathcal S_d^{-1}[F]$ for the basis function $F(s) =(c^2+s^2)^{-\nicefrac{1}{2}}$ of the inverse multiquadric kernel.

\begin{proposition} \label{prop:imq}
For $c>0$, $d \geq 3$ and $F(s) =(c^2+s^2)^{-\nicefrac{1}{2}}$, it holds that $F = \mathcal S_d[f]$ with $f(t)=c^{d-1 }(c^2+t^2)^{-\nicefrac{d}{2}}$ and
$
\mathcal F_1(f) (r)= \omega_{d-1} (c|r|)^{\nicefrac{(d-1)}{2}} K_{\nicefrac{(d-1)}{2}} (2\pi c|r|)
$.
\end{proposition}
The proof is given in Appendix~\ref{proof:imq}.
Note that $f=\mathcal S_d^{-1}[F]$ was computed for several other common basis functions $F$ in \cite[Table 1]{H2024} and \cite[Table 2]{HJQ2025}.

\paragraph{Inversion Method 3: Derivative Formula}

In \cite[Thm.\ 2.3]{RQS2025}, the authors prove for $f\in L^1_\mathrm{loc}([0,\infty))$ (if $d\geq 3$ odd) or $f\in L^p_\mathrm{loc}([0,\infty))$, $p\in(2,\infty)$, (if $d\geq3$ even) that $\mathcal S_d[f]\in C^{\lfloor d/2\rfloor-1}((0,\infty))$.
If $F\in C^{\lfloor d/2\rfloor}((0,\infty))$, we have that $\mathcal S_d$ admits the (unique) preimage
\begin{equation}\label{eq:RLFI_inv}
f(t)=\mathcal S_d^{-1}[F](t)=\frac{2t}{c_d \Gamma(\nicefrac{(d-1)}{2})} (D^{\nicefrac{(d-1)}{2}}_+G)(t^2) \quad \text{ with }\quad  G(t)\coloneqq F(\sqrt{t})\sqrt{t}^{d-2},
\end{equation}
where $D_+^\alpha$ denotes the fractional derivative of order $\alpha$.
For $d\geq3$ odd, we can rewrite \eqref{eq:RLFI_inv} with classical derivatives and then extend the result to weakly differentiable functions.

\begin{proposition}[Inversion of $\mathcal S_d$]\label{prop:Sd_inv_sum_exp}
Let $d\ge 3$ be odd and set $n\coloneqq \nicefrac{(d-1)}{2}$. Define
\begin{align}\label{eq:poly_expan}
\mathcal G_d[F]=\frac{2^nn!}{(2n)!}\sum_{k=0}^n a_{n, k} t^{k}F^{(k)}(t)\quad \text{ for all } t\ge 0,
\end{align}
where $a_{0,0}=a_{1,1}=1$ and $a_{m,k}\in \N$ is defined recursively for $ k=1,\ldots, m$ through
\begin{equation}\label{eq:ank_recursion}
a_{m,0}=(d{-}2m)a_{m-1,0}, \quad a_{m,m}=1\quad \text{and}\quad a_{m,k}=(d{-}2m{+}k)a_{m{-}1,k}{+}a_{m,k{-}1}.
\end{equation}
Then, the following holds true.
\begin{enumerate}[itemsep=0pt]
    \item[(i)] The inverse $\mathcal S_d^{-1}\colon C^n([0,\infty))\to C([0,1])$ is well-defined and given by \eqref{eq:poly_expan}.
    \item[(ii)] The inverse $\mathcal S_d^{-1}\colon H^n([0,1])\to L^2([0,1])$ is well-defined and given by \eqref{eq:poly_expan}.
    Its norm is bounded by 
    \begin{equation}
    C_d\coloneqq\left(\tfrac{2^nn!}{(2n)!}(a_{n,0}+\ldots,+a_{n,n})\right)^{\nicefrac12}.
    \end{equation}
\end{enumerate}
\end{proposition}
The proof is given in Appendix \ref{sec:ProofProp4}.
A direct consequence of \eqref{eq:poly_expan} is that $\mathcal S_d^{-1}[F](t)$ depends only locally on $F$.
In practice, using \eqref{eq:RLFI_inv} and \eqref{eq:poly_expan} is numerically unstable for large $d$ since it involves high-order derivatives of $F$ and the coefficients $a_{m,k}$ explode.

\subsection{Operator Norm}\label{sec:operator_norm}

Next, we investigate the operator norm of $\mathcal S_d \colon \mathcal G \to \mathcal H$, which is key to derive error estimates.
Clearly, the norm depends on the function spaces $\mathcal G$ and $\mathcal H$.
\begin{theorem}\label{thm:S_d_embeds}
For $d\ge 3$, the operator $\mathcal S_d$ in \eqref{eq:RLFI} is bounded on the following spaces:
\begin{enumerate}[itemsep=0pt]
    \item[(i)] $\mathcal S_d\colon L^p([0,1])\to L^p([0,1])$ is bounded for $1<p<\infty$ with \smash{$\|\mathcal S_d\|_\mathrm{op}\le\frac{c_dp}{p-1}$}.
    \item[(ii)] $\mathcal S_d\colon L^\infty([0,1])\to L^\infty([0,1])$ is bounded with $\|\mathcal S_d\|_\mathrm{op}=1$. 
    \item[(iii)] $\mathcal S_d\colon L^1([0,1])\to L^1([0,1], w)$ is bounded with $\|\mathcal S_d\|_\mathrm{op}\le c_d\|w\|_{L^1([0,1],\nicefrac{1}{s})}$, where $w\in L^1([0,1],\nicefrac{1}{s})$ is a positive weight.
    \item[(iv)] $\mathcal S_d\colon H^1([0,1])\to H^1([0,1])$, is bounded with $\|\mathcal S_d\|_\mathrm{op} \leq (2+\frac{c_d}{2})^{\nicefrac{1}{2}}$.
\end{enumerate}
\end{theorem}
The proof is given in Appendix~\ref{proof:S_d_embeds}.
As the following counterexample shows, part (i) of Theorem~\ref{thm:S_d_embeds} is false for $p=1$.
\begin{example}
Define $g\colon [0,1]\to \R$ as $g(t)=\ln(\nicefrac{\e}{t})^{-1}$ for $t\in (0,1]$ and $g(0)=0$.
Then $g'(t)= t^{-1}\ln(\nicefrac{\e}{t})^{-2}\ge 0$.
By the fundamental theorem of calculus, we get that $\int_0^1 g'(t)\d t = g(1)-g(0)=1$ and thus $g' \in L^1([0,1])$.
Since $\varrho_d(t)\ge 1$ on some interval $(0,\varepsilon)$ with $\varepsilon>0$ depending on $d$, see also Figure \ref{fig:varrho_d}, we obtain for any $s\in (0,1]$ that
\begin{equation}
    \mathcal S_d [g'](s)
    =\int_0^1 g'(ts)\varrho_d(t)\d t
    \ge \int_0^\varepsilon g'(ts) \d t
    = \frac{g(\varepsilon s)}{s}
    = \frac{1}{s \ln( \nicefrac{\e}{\varepsilon s}) }.
\end{equation}
Substituting $u=\ln(\nicefrac{\e}{\varepsilon s})$ with $u' = -s^{-1}$, we see that $\mathcal S_d[g'] \notin L^1([0,1])$ because
\begin{equation}
    \int_0^1 |\mathcal S_d[g'](s)|\d s
    \geq\int_0^1 \frac{1}{s\ln(\nicefrac{\e}{\varepsilon s})}\d s
    =\int_{\ln (\nicefrac{\varepsilon}{\e})}^\infty \frac{1}{u} \d u
    =\infty .
\end{equation}
\end{example}

\section{Slicing as Minimization Problem}\label{sec:optim}

Instead of trying to find $f$ analytically, we search for some approximation with $\mathcal S_d[f]\approx F$. 
By Remark~\ref{rem:01}, it suffices if $\mathcal S_d[f]$ approximates $F$ well on $[0,1]$.
Hence, we consider the minimization problem
\begin{equation}\label{eq:slicing_as_min}
\argmin_{f\in\mathcal G} \|\mathcal S_d[f]-F\|_{\mathcal H},
\end{equation}
where $\mathcal G$ is a function space on $[0,1]$ and $\mathcal H$ is a Hilbert space of functions on $[0,1]$ with inner product $\langle\cdot,\cdot\rangle_{\mathcal H}$.
To ensure that \eqref{eq:slicing_as_min} is well-defined, we require that $\mathcal S_d[f]\in\mathcal H$ for all $f\in\mathcal G$, see also Theorem \ref{thm:S_d_embeds}, and that $F\in\mathcal H$.
To proceed, we have to choose $\mathcal G$ and $\mathcal H$.

\paragraph{Choice of $\mathcal H$}
Regarding the Hilbert space $\mathcal H$, we consider two choices.
For $\mathcal H=L^2([0,1])$, the problem \eqref{eq:slicing_as_min} corresponds to minimizing the \emph{mean square error} between $\mathcal S_d[f]$ and $F$ on $[0,1]$.
For $\mathcal H=H^1([0,1])$, the objective in \eqref{eq:slicing_as_min} provides an upper bound on the \emph{worst case error} $\|\mathcal S_d[f]-F\|_{L^\infty([0,1])}$, since the Sobolev embedding theorem implies that
\begin{equation}
    \|\mathcal S_d[f]-F\|_{L^\infty([0,1])}\leq \sqrt{2}\|\mathcal S_d[f]-F\|_{H^1([0,1])}.
\end{equation}
Note that we cannot choose $\mathcal H=L^\infty([0,1])$ directly since this is not a Hilbert space.

\paragraph{Choice of $\mathcal G$}
As detailed in Section~\ref{sec:fastsum_1d}, computing the kernel sums in \eqref{eq:SlicingApprox} requires the Fourier coefficients of $f(|\cdot|)$.
Since $f(|\cdot|)$ is even by definition, the Fourier series on $[0,1]$ is equivalent to the cosine series
\begin{equation}\label{eq:cos_trafo}
f(t)=\mathcal C^{-1}[a](t)
=a_0+\sqrt{2}\sum_{k=1}^\infty a_k \cos(\pi k t )
\end{equation}
with the orthonormal system
\begin{equation}\label{eq:CosineBasis}
g_{0}(t)\coloneqq 1\quad \text{and}\quad g_{k}(t)\coloneqq \sqrt{2}\cos(\pi k t)
\end{equation}
and $a_k=\int_{0}^1 f(t) g_k(t)\d t$, see \cite[Chapter 1.2]{PPST2018} for an overview.
For truncated coefficients $a\in \R^{K+1}$, we write with abuse of notation $\mathcal C^{-1}[a]\coloneqq \mathcal C^{-1}[\tilde a]$, where $\tilde a_k\coloneqq a_k$ for $k=0,\ldots, K$ and $\tilde a_k=0$ for $k>K$.
Now, we choose $\mathcal G$ as the space of all functions with a $K$-term cosine transform, namely
\begin{equation}
    \mathcal G=\mathcal G_K\coloneqq\left\{f_a(t)\coloneqq\sum_{k=0}^{K-1}a_kg_k(t):a\in\R^{K}\right\}\subseteq C^\infty([0,1]).
\end{equation}
Inserting $\mathcal G_K$ into \eqref{eq:slicing_as_min}, we obtain the finite minimization problem
\begin{equation}\label{eq:min_forward}
    \hat a=\argmin_{a\in \R^{K}} \Vert \mathcal S_d[f_a]-F\Vert_{\mathcal H}^2.
\end{equation}
Even if the inverse $f=\mathcal S_d^{-1}[F]$ exists, the minimizer $\hat a$ does in general not coincide with the truncated cosine transform of $f$.
We give a counterexample in Appendix~\ref{app:not_truncated}.

\paragraph{Regularization}
Since $\mathcal S_d$ is ill-conditioned, we deploy Tikhonov regularization for $f_a$, leading to the regularized minimization problem
\begin{equation}\label{eq:min_forward_reg}
    \hat a=\argmin_{a\in \R^{K}} \Vert \mathcal S_d[f_a]-F\Vert_{\mathcal H}^2+\tau^2\|f_a\|^2_{\mathcal G}.
\end{equation}
Natural choices for $\mathcal G$ in \eqref{eq:min_forward_reg} are $L^2([0,1])$ and $H^1([0,1])$.
Since $\mathcal G_K \subseteq H^1([0,1]) \subseteq L^2([0,1])$, we obtain that $\|f_a\|_{\mathcal G}$ is well-defined in both cases.
Moreover, using that the cosine transform $\mathcal C \colon L^2([0,1]) \to \ell^2$ is an isometry together with the derivative rule for the cosine transform (see \cite[Chapter 1.2]{PPST2018}), we obtain that
\begin{equation}\label{eq:reg}
    \|f_a\|^2_{L^2([0,1])}=\|a\|_2^2,\quad\text{and}\quad \|f_a\|^2_{H^1([0,1])}=\sum_{k=0}^{K-1} (1+\pi^2k^2)|a_k|^2.
\end{equation}
Hence, we get $\|f_a\|_{\mathcal G}=\Vert D a \Vert_2$, where the diagonal matrix $D\in \R^{K\times K}$ has entries $D_{k,k} = 1$ or $D_{k,k} = \sqrt{(1+\pi^2k^2)}$, respectively. 
We will see in Section~\ref{sec:error_bounds} that the regularization in \eqref{eq:min_forward_reg} plays an important role for deriving bounds on the slicing error.
Next, we present two methods for solving \eqref{eq:min_forward_reg} numerically.

\subsection{Solution Method 1: Approximation in the Spatial Domain}\label{sec:spatial}

Defining $h_{k}\coloneqq \mathcal S_d[g_{k}] \in C^\infty([0,1])$ and using the linearity of $\mathcal S_d$, we rewrite the first term in \eqref{eq:min_forward_reg} as
\begin{equation}\label{eq:SpatialFit}
\|\mathcal S_d[f_a]-F\|_{\mathcal H}^2 = \Vert  a_0h_{0}+\ldots+a_Kh_{K-1}-F\Vert_{\mathcal H}^2.
\end{equation}
Since $\mathcal S_d$ commutes with dilations (Remark~\ref{rem:dilations}), we further obtain that $h_k(t)=\mathcal S_d[g_k](t)=\sqrt{2}\mathcal S_d[\cos](\pi k t)$ for $k\ge 1$.
Next, we establish a link between the principle function $\eta_d \coloneqq \mathcal S_d[\cos]$ (see Figure~\ref{fig:eta_d} for an illustration) and the confluent hypergeometric limit function 
\begin{equation}\label{eq:bessel_0F1}
  {}_0F_1(\nu+1;-\nicefrac{z^2}{4}) = \Gamma(\nu+1) (\nicefrac{2}{z})^\nu J_\nu(z)\quad z\in \R,
\end{equation}
where $J_\nu$ is the Bessel function of first kind with order $\nu>0$, see \cite[eq.\ 10.16.9]{dlmf}.
\begin{proposition}\label{prop:g}
For $d\geq 3$ let $\eta_d\coloneqq \mathcal S_d[\cos]$.
Then it holds for every $s\ge 0$ that
\begin{enumerate}[itemsep=0pt]
    \item[(i)] \label{en:g_bessel} 
    $\eta_d(s)
    = \Gamma(\frac{d}{2}) (\frac{2}{s})^{\nicefrac{d}{2}-1}J_{\nicefrac{d}{2}-1}(s)
    = {_0}F_1(\frac{d}{2}, -\frac{s^2}{4}) = \sum_{k=0}^\infty \frac{(-1)^k \Gamma(\frac{d}{2})}{\Gamma(K)\Gamma(k+\frac{d}{2})}(\frac{s}{2})^{2k}
    $;
    \item[(ii)] $|\eta_d(s)|\le 1$  with  $\eta_d(0)=1$;
    \item[(iii)] $\eta_d'(s)=-\frac{s}{d}\eta_{d+2}(s)$ and $|\eta_d'(s)|\le \frac{c_{d+2}}{d}$ with $\eta_d'(0)=0$.
\end{enumerate}
\end{proposition}
\begin{figure}[t]
    \centering
    \includegraphics[width=1\linewidth]{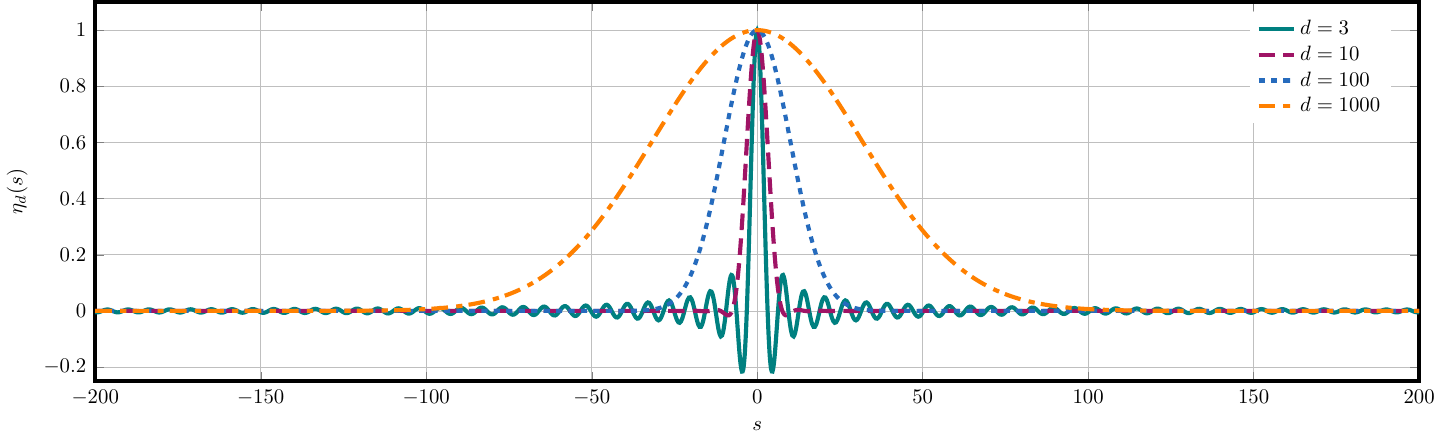}
    \caption{The function $\eta_d$ for $d\in \{3,10,100,1000\}$ on $x\in [-200,200]$.}
    \label{fig:eta_d}
\end{figure}
The proof is given in Appendix~\ref{proof:g}.
Note that $\tfrac{c_{d+2}}{d}\le 1$ for $d\ge 3$ and $\tfrac{c_{d+2}}{d}\to 0$ as $d\to \infty$.
Hence, Proposition \ref{prop:g} implies that $\eta_d$ is bounded and 1-Lipschitz continuous. 
Still, its numerical evaluation is challenging since both the power series and the Bessel function in part (i) converge slowly for large $s$.
Instead, we discretize the integral in $\mathcal S_d$ by quadrature rules.
It remains to discretize the norm in \eqref{eq:SpatialFit}.
Here, we focus on the case $\mathcal H=L^2([0,1])$.
Analogous formulas can be derived for $\mathcal H=L^2([0,1], w)$.

\paragraph{Numerical Approximation for $\mathcal H=L^2([0,1])$.}

For $L\in \mathbb N$, we denote by $\{t_l\}_{l=1}^L\subset [0,1]$ the (linearly transformed) Gauss--Legendre quadrature points and by $\{v_l\}_{l=1}^L$ the corresponding weights for approximating $\int_{0}^1 f(t)\d t \approx \sum_{l=1}^L v_lf(t_l)$. 
Using this quadrature rule, which is exact for polynomials of degree up to $2L-1$, we approximate $h_k(s_l)$ at the points $\{s_l\}_{l=1}^L$ with $s_l=t_l$ as
\begin{align}\label{eq:num_int_H}
    h_k(s_l) = \int_0^1 g_{k}(s_lt)\varrho_d(t)\d t\approx \sum_{j=1}^L v_j g_{k}(s_lt_j)\varrho_d(t_j)\eqqcolon \hat h_{k,l}.
\end{align}
Now, we can approximate the norm in \eqref{eq:SpatialFit} using Gaussian quadrature and \eqref{eq:num_int_H} as
\begin{equation}\label{eq:MatH}
\|\mathcal S_d[f_a]-F\|_{\mathcal H}^2 \approx \sum_{l=1}^L \biggl(\sum_{k=1}^K h_k(s_l) a_k - F(s_l)\biggr)^2 v_l \approx \sum_{l=1}^L \biggl(\sum_{j=1}^K \hat h_{k,l} a_k - F(s_l)\biggr)^2 v_l.
\end{equation}
In matrix-vector form, the approximation \eqref{eq:MatH} can be rewritten as $\Vert \hat H^\top a -  \hat b \Vert^2$, where $v=(v_1,...,v_L)^\tT$, $\hat H_{k,l} \coloneqq \hat h_{k,l} \sqrt{v_l}$ and $\hat b\coloneqq [\sqrt{v_l}F(s_l)]_{l=1}^L\in \R^{K}$.
Then, we solve this least squares problem with the same regularization as for \eqref{eq:min_forward_reg} using the built-in solver \texttt{torch.linalg.lstsq} of PyTorch.
The method is summarized in Algorithm \ref{alg:bessel}.

\begin{algorithm}[t]
\caption{Spatial Domain}\label{alg:bessel}
\textbf{Input}: $F\colon [0,1]\to \R$, dimension $d\in \N$.\\
\textbf{Hyperparameters}: Number of cosine coefficients $K$ and quadrature points $L$, regularization strength $\tau> 0$, regularization matrix $D\in \R^{K\times K}$.
\begin{algorithmic}[1]
\State Get quadrature points $t_1,\ldots, t_L\in [0,1]$ and weights $v_1,\ldots, v_L\in \R$.
\State Build $\hat H \in \R^{K\times L}$ as $\hat H_{k,l}=\hat h_{k,l}\sqrt{ v_l}$. 
\State Compute $\hat b= [\sqrt{v_l}F(s_l)]_{l=1}^L\in \R^{L}$.
\State \label{alg1:line:reg_prob}Compute $\hat a=\argmin_a \|\hat H^\top a-\hat b\|_2^2 + \tau^2 \Vert D a \Vert_2^2$.
\end{algorithmic}
\end{algorithm}

\begin{remark}
The integral in \eqref{eq:num_int_H} involves the weight function $\varrho_d$, which also appears for the Gegenbauer polynomials of order $\nicefrac{(d-2)}{2}$.
Therefore, we could replace the Gauss-Legendre quadrature in \eqref{eq:num_int_H} by Gauss-Gegenbauer quadrature.
Numerically, the accuracy of the Gauss-Legendre quadrature was sufficient for our purpose.
Moreover, computing the zeros of Gegenbauer polynomials for large $d$ is numerically challenging.
Thus, we stick with Gauss-Legendre quadrature points.
\end{remark}

\subsection{Solution Method 2: Approximation in the Frequency Domain}

Our second method for solving \eqref{eq:min_forward_reg} shifts the problem into the range of the cosine transform $\mathcal C$.
To this end, we assume that $\mathcal H\subseteq L^2([0,1])$ and define the inner product on $\mathcal C(\mathcal H)\subseteq \ell^2$ by
\begin{equation}
    \langle a, b\rangle_{\ell(\mathcal H)}\coloneqq \langle \mathcal C^{-1}a, \mathcal C^{-1}b\rangle_{\mathcal H}\quad a,b\in \mathcal C(\mathcal H).
\end{equation}
For $\mathcal H=L^2([0,1])$ and $\mathcal H=H^1([0,1])$, the resulting norms are given in \eqref{eq:reg}.
By construction of $C(\mathcal H)$, the cosine transform $\mathcal C\colon\mathcal H\to\mathcal C(\mathcal H)$ is an isometry such that we can rewrite the minimization problem \eqref{eq:min_forward_reg} as
\begin{equation}\label{eq:min_forward_reg_freq}
    \hat a=\argmin_{a\in \R^{K}} \Vert \mathcal (\mathcal C \circ S_d \circ \mathcal C^{-1})[a]-\mathcal C[F]\Vert_{\ell(\mathcal H)}^2+\tau^2\|f_a\|^2_{\mathcal G}.
\end{equation}
We are now interested in the operator $S\coloneqq \mathcal C \circ S_d \circ \mathcal C^{-1}$, which is related with $\mathcal S_d$ as depicted in the following commuting diagram.
\begin{center}
\begin{tikzcd}[row sep=huge, column sep=huge]
L^2([0,1]) \arrow[thick,r, "\mathcal S_d"] \arrow[thick, d, " \mathcal C"'] 
  & L^2([0,1]) \arrow[thick, d, "\mathcal C"] \\
\ell^2 \arrow[thick, dashed, r, "S"'] 
  & \ell^2
\end{tikzcd}
\end{center}
In particular, $S$ can be represented by the display matrix
$S_{j,k} \coloneqq \langle g_{j}, \mathcal S_d[g_{k}]\rangle_{L^2([0,1])}$ for $j,k\in \N$
in the sense that $b=Sa$ is given by $b_j=\sum_{k=0}^\infty S_{j,k}a_k$, where $\{g_k\}_{k \in \N}$ is the cosine basis \eqref{eq:CosineBasis}.
By truncating $S$ and $\mathcal C[F]$ in \eqref{eq:min_forward_reg_freq} to $J$ coeffcients, we obtain the discrete problem \begin{equation}\label{eq:Sa=b}
    \hat a=\argmin_{a\in \R^{K}} \|Sa-b\|^2_{\ell(\mathcal H)}+\tau^2\|f_a\|^2_{\mathcal G}
\end{equation}
with $S\in \R^{J\times K}$ and $b_k = \mathcal C[F]_k$ for $k=0,\ldots, J$.
To compute the $S_{j,k}$ numerically, we use the representation from the following proposition. 

\begin{proposition}\label{prop:compute_s}
Let $d\ge 3$.
Then it holds 
\begin{equation}\label{eq:Sdjk}
S_{j,k} = \mathcal S_d[\sinc(\cdot+j)+\sinc(\cdot-j)](k)\quad \text{for}\quad j,k\ge 1.
\end{equation}
Moreover, we have $S_{j,0}=\delta_{j,0}$ and $S_{0,k}= \frac{1}{\sqrt{2}} \mathcal S_d[\sinc](k)$.
\end{proposition}
\begin{proof}
    For $j,k\ge 1$ it holds
\begin{align}\label{eq:SFirstInt}
S_{j,k}
&
=\langle g_{j}, h_{k} \rangle_{L^2([0,1])} 
=\int_0^1  \int_0^1 \sqrt{2}\cos (\pi kts)\varrho_d(t) \d t \sqrt{2}\cos(\pi j s)\d s\notag\\&
=2\int_0^1\varrho_d(t)\int_0^1 \cos(\pi j s) \cos(\pi kts)\d s\, \d t.
\end{align}
Using the angle sum identities for the cosine, we further get
\begin{align}
S_{j,k} &
=2\int_0^1 \varrho_d(t) \frac{1}{2}\int_0^1 \cos(\pi (kt+j)s)+\cos(\pi (kt-j)s)\d s\, \d t \notag\\&
=\int_0^1 \varrho_d(t) \left[\frac{\sin(\pi (kt+j)s)}{\pi (kt+j)} +\frac{\sin(\pi (kt-j)s)}{\pi (kt-j)} \right]_0^1\d t \notag\\&
=\int_0^1 \varrho_d(t) \left[\frac{\sin(\pi (kt+j))}{\pi (kt+j)} +\frac{\sin(\pi (kt-j))}{\pi (kt-j)} \right]\d t\notag\\&
=\mathcal S_d[\sinc(\cdot +j)+\sinc(\cdot -j)](k).
\end{align}
\end{proof}

In practice, we compute the integral transform in \eqref{eq:Sdjk} by Gauss--Legendre quadrature as in Section~\ref{sec:spatial}.
We summarize the method in Algorithm~\ref{alg:fourier}.
The least squares problem in the last step is solved using the \texttt{torch.linalg.lstsq} solver of PyTorch.

\begin{algorithm}[t]
\caption{Frequency Domain}\label{alg:fourier}
\textbf{Input}:  $F\colon [0,1]\to \R$, dimension $d\in \N$.\\
\textbf{Hyperparameters}: Number of cosine coefficients $K$ in domain and $J$ in range, quadrature points $L$, regularization parameter $\tau > 0$, regularization matrix $D\in \R^{K\times K}$.
\begin{algorithmic}[1]
\State Get Gaussian quadrature points $t_1,\ldots, t_L\in [0,1]$ and weights $v_1,\ldots, v_L\in \R$.
\State Assemble $S\in \R^{J\times K}$ as in \eqref{eq:Sdjk} via Gaussian quadrature. 
\State Compute $b = \mathcal C[F]\in \R^{J}$ via FFT.
\State Solve $\hat a=\argmin_a \| Sa-b\|_{\ell(\mathcal H) }^2 + \tau^2 \Vert D a \Vert_2^2$.
\end{algorithmic}
\end{algorithm}

\section{Error Estimates}\label{sec:error_bounds}

Next, we study the \emph{approximation error} in \eqref{eq:slicing}.
For simplicity, we only consider directions $\xi_1,...,\xi_P$ that are independently drawn from $\mathcal U_{\mathbb S^{d-1}}$.
Empirically, QMC designs lead to even better rates \cite{HJQ2025}.
We denote for $f,F\in C([0,1])$ the \emph{mean square error} at $x\in \R^d$ as
\begin{equation}\label{eq:ApproximationError}
\mathcal E_{d,P}[f, F](\|x\|)^2\coloneqq \mathbb E_{\xi_1,\ldots, \xi_P\sim \mathcal U_{\mathbb S^{d-1}}} \Biggl[ \biggl(\frac{1}{P} \sum_{p=1}^P f(|\langle \xi_p, x\rangle|) - F(\|x\|)\biggr)^2\Biggr].
\end{equation}
Previous works \cite{H2024,HJQ2025} analyzed the error \eqref{eq:ApproximationError} for $F= \mathcal S_d[f]$, which leads to the \emph{slicing error} $\mathcal E_{d,P}\bigl[f,\mathcal S_d[f]\bigr](\|x\|)$.
The authors of \cite{HJQ2025} express this error via the variance $\mathbb V_d[f](\|x\|) \coloneqq \mathbb E_{\xi\sim \mathcal U_{\mathbb S^{d-1}}} [( f(|\langle \xi,x\rangle|)- \mathcal S_d[f](\|x\|))^2]$ of $f$ using Bienaymé’s identity \cite{B1853,HS1977} as
\begin{align}\label{eq:SliceAndVariation}
    \mathcal E_{d,P}\bigl[f,\mathcal S_d[f]\bigr](\|x\|)^2 = \frac{\mathbb V_d[f](\|x\|)}{P}.
\end{align}
In practice, we only compute an estimate $\hat f$ of $f$ such that also $\mathcal S_d[\hat f]$ and $F$ differ.
This mismatch is measured by the \emph{forward error} defined as \begin{equation}\label{eq:ForwardError}
    \bigl|\mathcal S_d[\hat f](\|x \|)-F(\|x\|)\bigr|.
\end{equation}
As a direct consequence of the bias-variance decomposition, we can split the mean square error \eqref{eq:ApproximationError} into the forward and slicing error.
\begin{theorem}\label{thm:error_dcomp}
For any $\hat f, F \in \mathcal C([0,1])$ and $x \in \R^d$ with $\|x\| \leq 1$, it holds that
\begin{equation}\label{eq:err_dcomp}
    \mathcal E_{d,P}[\hat f, F](\|x\|)^2 = \underbrace{|F(\|x\|)-\mathcal S_d[\hat f](\|x\|)|^2}_{\text{forward error \eqref{eq:ForwardError}}} + \underbrace{\mathbb V_d[\hat f](\|x\|)P^{-1}}_{\text{slicing error \eqref{eq:SliceAndVariation} }}.
\end{equation}
\end{theorem}
\begin{proof}
Let $x\in \R^d$ with $\|x\|\le 1$ and
recall that $\mathcal S_d[\hat f](\|x\|)=\mathbb E_{\xi \sim \mathcal U_{\mathbb S^{d-1}}} [\hat f(|\langle x, \xi \rangle |)]$.
Using the bias variance decomposition and Bienaymé’s identity  \cite{B1853,HS1977}, we obtain that
\begin{align}
\mathcal E_{d,P}[\hat f,F](\|x\|)^2&=\mathbb E\biggl[\frac{1}{P} \sum_{p=1}^P \hat f(|\langle \xi_p, x\rangle|)-F(\|x\|)\bigr)\biggr]^2+\mathrm{Var}\biggl[\frac{1}{P} \sum_{p=1}^P \hat f(|\langle \xi_p, x\rangle|)\biggr]\notag\\
&=|F(\|x\|)-\mathcal S_d[\hat f](\|x\|)|^2+\mathbb V_d[\hat f](\|x\|)P^{-1}.
\end{align}
\end{proof}

Now, we establish the link to the regularization in the optimization problem \eqref{eq:min_forward_reg}.
Basically, the idea is that the slicing error with respect to $f_a\in \mathcal G_K$ cannot be larger than the squared norm of $f_a$.

\begin{lemma}\label{lem:stupid_bound}
Let $a\in\R^{K}$ and $f_a=\mathcal C^{-1}[a]$. Moreover, let $w\in L^1([0,1],1/s)$ be a weight with $C_w\coloneqq \|w\|_{L^1([0,1],\nicefrac{1}{s})}<\infty$. Then, the variance of $f_a$ can be bounded as
\begin{align}\label{eq:VarBound_L1}
    \|\mathbb V_d[f_a]\|_{L^1([0,1],w)}&\le c_dC_w \|a\|_2^2= c_dC_w\|f_a\|_{L^2([0,1])}^2\\
    \|\mathbb V_d[f_a]\|_{L^\infty([0,1])}\hspace{5pt}&\le 2\|a\|_1^2\le 3\|f_a\|_{H^1([0,1])}^2.\label{eq:VarBound_Linf}
\end{align}
\end{lemma}
\begin{proof}
Denote by $F_a=\mathcal S_d[f_a]$ and let $x\in\R^d$ with $s=\|x\|\leq 1$.
Then, the variance can be rewritten as
\begin{align}
\mathbb V_d[f_a](s)
&
=\mathbb E_{\xi\sim \mathcal U_{\mathbb S^{d-1}}}[ (f_a(|\langle \xi,x\rangle|)-F_a(s))^2]
= \mathcal S_d[f_a^2]-2\mathcal S_d[f_a](x)F_a(s)+F_a(s)^2
\notag\\&
=\mathcal S_d[f_a^2](s) - F_a(s)^2 \leq S_d[f_a^2](s).\label{eq:EstVariance}
\end{align}
Now, Theorem \ref{thm:S_d_embeds} implies that
\begin{align}
\|\mathbb V_d[f_a]\|_{L^1([0,1],w)}
\le \|\mathcal S_d[f_a^2]\|_{L^1([0,1],w)}
\le c_d C_w \|f_a^2\|_{L^1([0,1])}
= c_d C_w \|f_a\|_{L^2([0,1])}^2.
\end{align}
Regarding \eqref{eq:VarBound_Linf}, we first bound $\vert f_a \vert$.
By \cite[1.421 4.]{GR2015} it holds that
\begin{equation}
\sum_{k=0}^K\frac2{1+\pi^2k^2}\leq \sum_{k=0}^\infty\frac2{1+\pi^2k^2}=\coth(1)+1\leq 3.
\end{equation}
Therefore, we get by Hölder's inequality that
\begin{align}
    |f_a(t)|^2
    \le 2\biggl(\sum_{k=0}^K |a_k|\biggr)^2 \le \sum_{k=0}^K \frac{2}{1+\pi^2k^2} \sum_{k=0}^K (1+\pi^2k^2) |a_k|^2
    =3\|f_a\|_{H^1([0,1])}^2.\label{eq:EstfK}
\end{align}
Inserting \eqref{eq:EstfK} into \eqref{eq:EstVariance}, we finally obtain \eqref{eq:VarBound_Linf}.
\end{proof}

Hence, the regularization \eqref{eq:reg} ensures that we can upper bound the slicing error.
Putting Theorem \ref{thm:error_dcomp} and Lemma \ref{lem:stupid_bound} together, we obtain that the mean square error \eqref{eq:ApproximationError} can be bounded by the objective value of the regularized problem \eqref{eq:min_forward_reg}.

\begin{corollary}\label{cor:to_reg_prob}
Let $a\in\R^{K}$ and $f_a=\mathcal C^{-1}[a]$. Moreover, let $w\in L^1([0,1],1/s)$ be a weight with $C_w\coloneqq \|w\|_{L^1([0,1],\nicefrac{1}{s})}<\infty$.
Then, it holds that
\begin{align}
\|\mathcal E_{d,P}[f_a,F]\|_{L^2([0,1],w)}^2
&\le c_dC_wP^{-1}\|f_a\|_{L^2([0,1])}^2+\|F-\mathcal S_d[f_a]\|_{L^2([0,1],w)}^2 \label{eq:L2_make_reg}\\
\|\mathcal E_{d,P}[f_a,F]\|_{L^\infty([0,1])}^2
&\le 3P^{-1}\|f_a\|_{H^1([0,1])}^2+2\|F-\mathcal S_df_a\|_{H^1([0,1])}^2.\label{eq:Linf_make_reg}
\end{align}
\end{corollary}
\begin{proof}
Both statements follow directly from Theorem~\ref{thm:error_dcomp}, Lemma~\ref{lem:stupid_bound} and the Sobolev embedding $\|f\|_{L^\infty([0,1])}\leq \sqrt{2}\|f\|_{H^1([0,1])}$ for $f\in H^1([0,1])$.
\end{proof}

Finally, we investigate, how well $F$ can be approximated by $\mathcal S_d[f_a]$ for $a\in\R^{K}$ depending on the number of cosine coefficients $K$.
In other words, we are interested in the minimal value that can be attained in the (unregularized) problem \eqref{eq:min_forward}.

\begin{proposition}\label{prop:forward_error}
Assume that $F=\mathcal S_d[f]$ for some $f\colon[0,1]\to\R$ with $f(|\cdot|)\in H^r([-1,1])$ for $r\ge 1$.
Let $a\in\R^{K}$ be the minimizer of \eqref{eq:min_forward} and $f_a=a_0g_0+\ldots +a_Kg_K$.
\begin{enumerate}
    \item[(i)]  If $\mathcal H=L^2([0,1])$, there is a constant $C>0$ independent of $K$ such that it holds $\|F-\mathcal S_df_a\|_{L^2} \le C K^{-\frac{r}{1+ 2r/(d-1)}}$ for all $K\in \N$.
    \item[(ii)] If $\mathcal H=H^1([0,1])$, there is a constant $C>0$ independent of $K$ such that it holds $\|F-\mathcal S_df_a\|_{H^1} \le C K^{-\frac{r-1}{1+ 2r/(d-3)}}$ for all $K\in \N$.
\end{enumerate}

\end{proposition}
\begin{proof}
For any $0<\varepsilon<\nicefrac{1}{2}$, there exists a cutoff function $\varphi_\varepsilon\in \mathcal C^\infty_c([0,1))$ with $\varphi_\varepsilon(t)=1$ for $t\in [0,1-\varepsilon]$ and  \smash{$|\varphi_\varepsilon^{(k)}(t)|\le \tilde c_k \varepsilon^{-k}$} with  constants $\tilde c_k>0$, see \cite[Thm.\ 1.4.1]{H2003}.
The even extension of $f_\varepsilon\coloneqq f\varphi_\varepsilon$ is $r$ times periodic weakly differentiable on $[-1,1]$.
For $k=0,\ldots,K$, we set $\tilde a_k \coloneqq \mathcal C[f_\varepsilon]_k$ and $f_{
\tilde a}
\coloneqq \tilde a_0g_0+\ldots+\tilde a_Kg_K$.
Then, it holds that
\begin{align}\label{eq:cut-off}
\|\mathcal S_d f_a -F\|_{\mathcal H}
&
\le \|\mathcal S_d f_{\tilde a} - \mathcal S_d f\|_{\mathcal H}
\le \|\mathcal S_d [f_{\tilde a} - f_\varepsilon]\|_{\mathcal H} + \|\mathcal S_d[f_\varepsilon - f]\|_{\mathcal H}.
\end{align}

(i) For $\mathcal H=L^2([0,1])$, the first summand on the right hand side of \eqref{eq:cut-off} can be bounded by Theorem \ref{thm:S_d_embeds} and \cite[Thm.~1.1]{CQ1982} as
\begin{equation}\label{eq:EstFourierSeries}
\|\mathcal S_d [f_{\tilde a} - f_\varepsilon]\|_{L^2}
\le 2c_d \|f_{\tilde a}-f_\varepsilon]\|_{L^2}
\le 2c_d \|f_\varepsilon^{(r)}\|_{L^2}(\pi K)^{-r}.
\end{equation}
Using the product rule, we further obtain with $\tilde c_f\coloneqq \sum_{k=0}^r \binom{r}{k} \tilde c_k \|f^{(r-k)}\|_{L^2}$ that
\begin{align}\label{eq:EstDerivative}
 \|f_\varepsilon^{(r)}\|_{L^2}
 &
 \le  \sum_{k=0}^r \binom{r}{k}\|\varphi_\varepsilon^{(k)} \|_{L^2}\|f^{(r-k)}
 \|_{L^2}
 \le \tilde c_f \,\varepsilon^{-r}.
\end{align}
Regarding the second summand in \eqref{eq:cut-off}, we have that any $g\in L^2([0,1])$ which vanishes on $[0,1-\varepsilon]$ satisfies by Jensen's inequality
\begin{align}\label{eq:EstPartial_I}
     \|\mathcal S_d[g]\|_{L^2}^2
     &
    = \int_0^1 \Bigl( \int_0^1 g(ts)\varrho_d(t)\d t\Bigr)^2 \d s
    = \int_0^1 \Bigl( \int_{1-\varepsilon}^1 g(ts)\varrho_d(t)\d t\Bigr)^2 \d s
    \\&
    \le c_d^2(2\varepsilon - \varepsilon^2)^{d-3}\int_0^1 \varepsilon^2 \Bigl( \int_{1-\varepsilon}^1 g(ts)\frac{\d t}{\varepsilon}\Bigr)^2\d s
    \le c_d^2(2\varepsilon)^{d-3} \varepsilon \int_0^1 \int_{1-\varepsilon}^1g^2(ts)\d t\d s .\notag 
\end{align}
From Fubini's theorem, we further get with $-\ln(1-\varepsilon)\le 2\varepsilon$ for $\varepsilon<\nicefrac{1}{2}$ that
\begin{equation}\label{eq:EstPartial_II}
    \int_0^1 \int_{1-\varepsilon}^1g^2(ts)\d t\d s \leq \int_{1-\varepsilon}^1 \frac{1}{t}\int_0^t g^2(s)\d s \d t  
    \le  \|g\|_{L^2}^2 \int_{1-\varepsilon}^1 \frac{1}{t} \d t
    \le 2\varepsilon \|g\|_{L^2}^2.
\end{equation}
By inserting $g=f_\varepsilon-f$ into \eqref{eq:EstPartial_I} and \eqref{eq:EstPartial_II}, we finally obtain that
\begin{align}\label{eq:sec_sumL2}
 \|\mathcal S_d[f_\varepsilon-f]\|_{L^2}^2
 \le c_d^2(2\varepsilon)^{d-2}\varepsilon \|f(\varphi_\varepsilon-1)\|_{L^2}^2
 \le c_d^22^{d-2} \|f\|_{L^2}^2 \varepsilon^{d-1}.
\end{align}
Now, for $\varepsilon\coloneqq K^{-\frac{2r}{d-1+2r} }$, we combine \eqref{eq:cut-off}, \eqref{eq:EstFourierSeries}, \eqref{eq:EstDerivative} and \eqref{eq:sec_sumL2} to get
\begin{align}
   \|\mathcal S_df_a-F\|_{L^2}
   &
   \le 2c_d\tilde c_f(\pi K)^{-r} K^{-r\frac{-2r}{d-1+2r}}+c_d  2^{\frac{d-2}{2}} \|f\|_{L^2} K^{-\frac{2r}{d-1+2r} \frac{d-1}{2}}\notag
    \\&
    = c_d\bigl(2\tilde c_f\pi^{-r}+2^{\frac{d-1}{2}} \|f\|_{L^2}\bigr) K^{-\frac{r(d-1)}{d-1+2r}}.
\end{align}

(ii) For $\mathcal H=H^1([0,1])$, we estimate the first summand in \eqref{eq:cut-off} with Theorem \ref{thm:S_d_embeds} as
\begin{equation}\label{eq:decay_Fourier_H1}
\|\mathcal S_d [f_{\tilde a} - f_\varepsilon]\|_{H^1}
\le  \bigl(2+\tfrac{c_d}{2}\bigr)^{\nicefrac{1}{2}}\|f_{\tilde a} - f_\varepsilon\|_{H^1}\le  (4+c_d)^{\nicefrac{1}{2}}\|f_\varepsilon^{(r)}\|_{L^2} (\pi K)^{1-r}.
\end{equation}
Next, we bound the second summand of \eqref{eq:cut-off} as
\begin{equation}
 \|\mathcal S_d[f_\varepsilon-f]^\prime\|_{L^2}^2
=\int_0^1 \Bigl( \int_0^1 t(f_\varepsilon^\prime -f^\prime )(ts)\varrho_d(t)\d t\Bigr)^2 \d s
\le \|\mathcal S_d[|f_\varepsilon^\prime-f^\prime|]\|_{L^2}^2.
\end{equation}
Using \eqref{eq:EstPartial_I} and \eqref{eq:EstPartial_II} with $g=\vert f_\varepsilon^\prime -f^\prime \vert$, it holds that
\begin{align}
\|\mathcal S_d[|f_\varepsilon^\prime-f^\prime|]\|_{L^2}^2 &
\le c_d^2(2\varepsilon)^{d-2}\varepsilon \|f_\varepsilon^\prime-f^\prime\|_{L^2}^2
= c_d^2(2\varepsilon)^{d-2}\varepsilon \| f^\prime\varphi_\varepsilon+f\varphi_\varepsilon^\prime-f^\prime\|_{L^2}^2 \notag \\& 
\le c_d^2(2\varepsilon)^{d-2}\varepsilon \bigl(\|f^\prime \|_{L^2} + \tfrac{\tilde c_1}{\varepsilon}\|f\|_{L^2}\bigr)^2.
\end{align}
Since $\tilde c_1 \geq 1 > \varepsilon$ by the mean value theorem, we can further estimate 
\begin{align}\label{eq:sec_sumH1}
\|\mathcal S_d[f_\varepsilon-f]\|_{H^1}^2
& \le c_d^2(2\varepsilon)^{d-2}\bigl(\varepsilon\|f\|^2_{L^2}+\varepsilon (\|f^\prime\|_{L^2}+\tfrac{\tilde c_1}{\varepsilon}\|f\|_{L^2})^2\bigr)\notag 
\\&
\le c_d^2(2\varepsilon)^{d-2}\bigl(\tfrac{\varepsilon^2 + 2\tilde c_1^2}{\varepsilon}\|f\|^2_{L^2}+2 \varepsilon \|f^\prime\|_{L^2}^2\bigr)
\le 
(\tilde c_1c_d)^22^{d}\varepsilon^{d-3}\|f\|_{H^1}^2.
\end{align}
Now, for $\varepsilon = K^{\frac{2-2r}{d-3+2r} }$, we ultimately obtain from \eqref{eq:cut-off}, \eqref{eq:EstDerivative}, \eqref{eq:decay_Fourier_H1} and \eqref{eq:sec_sumH1} that
\begin{align}
   \|\mathcal S_df_a-F\|_{H^1}
   &
   \le (4+c_d)^{\nicefrac{1}{2}}\tilde c_f  (\pi K)^{1-r}K^{-\frac{2-2r}{d-3+2r} r} + \tilde c_1c_d2^{\nicefrac{d}{2}}K^{\nicefrac{\frac{2-2r}{d-3+2r} (d-3)}{2}}\|f\|_{H^1}\notag
   \\&
   =\bigl((4+c_d)^{\nicefrac{1}{2}}\tilde c_f\pi^{1-r}+\tilde c_1c_d2^{\nicefrac{d}{2}}\|f\|_{H^1}\bigr)K^{\frac{(d-3)(1-r)}{d-3+2r}}.
\end{align}
\end{proof}

\begin{remark}\label{rem:forward_error}
If we  additionally assume that $f(|\cdot|)$ is $r$ times periodic differentiable in Proposition \ref{prop:forward_error}, then the proof shows that we can simplify the bounds as follows.
\begin{enumerate}[itemsep=0pt]
    \item[(i)] If $\mathcal H=L^2([0,1])$, it holds $\|F-\mathcal S_d[f_a]\|_{L^2} \le 2c_d\|f^{(r)}\|_{L^2} (\pi K)^{-r}$.
    \item[(ii)] If $\mathcal H =H^1([0,1])$, it holds $\|F-\mathcal S_d[f_a]\|_{H^1}\le (4+c_d)^{\nicefrac{1}{2}}\|f^{(r)}\|_{L^2} (\pi K)^{1-r}$.
\end{enumerate}
\end{remark}

Proposition \ref{prop:forward_error} implies that a small $K$ suffices if $F = \mathcal S_d[f]$ for a smooth $f$.
Moreover, using the Sobolev embedding and (ii), we can bound the worst case error $\|F-\mathcal S_d[f_a]\|_{L^\infty}$.

\section{Numerical Results}\label{sec:numerics}

We benchmark Algorithms~\ref{alg:bessel} and \ref{alg:fourier} with several choices of $\mathcal H$ and $\mathcal G$ as listed in Table~\ref{tab:methods}.
For this, we use the (kernel) basis functions $F$ from Table~\ref{tab:testfunc}, which are visualized in Figure~\ref{fig:functions}.
These have different characteristics.
The Bump is periodically smooth on $[-1,1]$ for $c\leq1$ with large derivatives around $\nicefrac{c}{2}$.
Gauss, IMQ and MQ are smooth on $[-1,1]$ but not periodically smooth.
For these choices, Proposition \ref{prop:forward_error} applies and we expect good approximation.
The TPS and Laplace are only (weakly) differentiable.
Lastly, the Log has an integrable singularity in $0$, but is not $H^1([-1,1])$.
These choices are expected to be more challenging.
If $f=\mathcal S_d^{-1}[F]$ is known, its cosine coefficients can be computed directly.
In this case, we report the associated slicing results as a competing method.
Our code for numerically computing the coefficients and performing the fast kernel summations is available on Github\footnote{\url{https://github.com/Nicolaj-Rux/slicing-inversion}\\\url{https://github.com/johertrich/simple_torch_NFFT.git}}.
Throughout, we use $L=2^{10}$ quadrature points, $K=2^8$ cosine coefficients in the domain and $J=2^{10}$ in the range.

\begin{figure}[p]
\begin{table}[H]
    \centering
\begin{tabular}{llcc}
\toprule
Method     & Algorithm                   & Range $\mathcal H$ & Domain $\mathcal G$ \\
\midrule
\texttt{S-L2-H1} & Alg.~\ref{alg:bessel} (\text{Spatial Domain})   & $L^2([0,1])$ & $H^1([0,1])$ \\
\texttt{F-L2-H1} & Alg.~\ref{alg:fourier} (\text{Frequency Domain})& $L^2([0,1])$ & $H^1([0,1])$ \\
\texttt{F-H1-H1} & Alg.~\ref{alg:fourier} (\text{Frequency Domain})& $H^1([0,1])$ & $H^1([0,1])$ \\
\bottomrule
\end{tabular}
    \caption{Abbreviations and setup for our numerical tests.}
    \label{tab:methods}
\end{table}

\begin{table}[H]
\centering
\scalebox{.77}{
\begin{tabular}{lcc}
\toprule 
\textbf{Kernel}& $F(s)$&$f(t)$
\\
\midrule 
Gauss &$\exp(-\frac{s^2}{2c^2})$&${_1}F_1(\tfrac{d}{2};\tfrac12;\tfrac{-t^2}{2c^2})$
\\[2ex]
Laplace &$\exp(-c |s|)$&$\sum_{n=0}^\infty \frac{(-1)^n c^n\sqrt{\pi}\Gamma(\frac{n+d}{2})}{n!\Gamma(\frac{d}{2})\Gamma(\frac{n+1}{2})} |t|^{n}$
\\[2ex]
Inverse Multi Quadric (IMQ)& $(c^2+s^2)^{-\nicefrac12}$ & $c^{d-1}(c^2+t^2)^{-\nicefrac{d}{2}}$ 
\\[2ex]
Thin Plate Spline (TPS) &$(cs)^2\log(|cs|)$&$d (ct)^2\log(|ct|)+\alpha_d (ct)^2$
\\[2ex]
Logarithmic (LOG)&$\log(|cs|)$& $\beta_d+\log(|ct|)$
\\[2ex]
Bump &$\exp(\tfrac{-c^2}{c^2-s^2})\chi_{|s|\le c}$& unknown
\\[2ex]

Multi Quadric (MQ) & $-\sqrt{c^2 + s^2}$ & unknown 
\\[2ex]
\bottomrule
\end{tabular}}
\caption{
Pairs $(F,f)$ satisfying the slicing formula \eqref{eq:slicing}.
Here, we require $c >0$ and define $ H_x\coloneqq \int_0^1 \frac{1-t^x}{1-t}\d t$, $\alpha_d\coloneqq \frac{d}{2}\big(H_{\nicefrac{d}{2}}-2+\log(4)\big)$ and $\beta_d\coloneqq {-}\int_0^1\log(r)\varrho_d(r)\d r$.
Reference: \cite[Table~1]{H2024}, \cite[Table~2]{HJQ2025} and Proposition~\ref{prop:imq}.
}

\label{tab:testfunc}
\end{table}

\begin{figure}[H]
\centering
  \begin{subfigure}[t]{0.48\textwidth}
    \centering
\includegraphics[width=\textwidth]{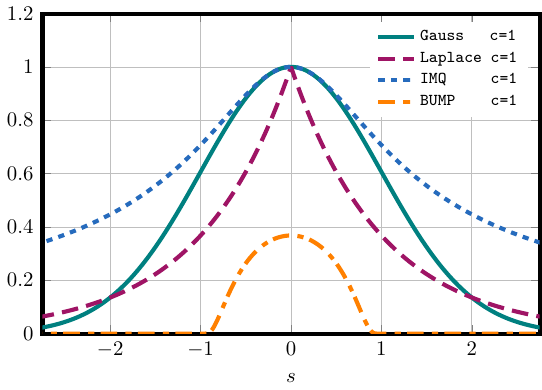}
  \end{subfigure}
  \hfill
  \begin{subfigure}[t]{0.48\textwidth}
    \centering
\includegraphics[width=\textwidth]{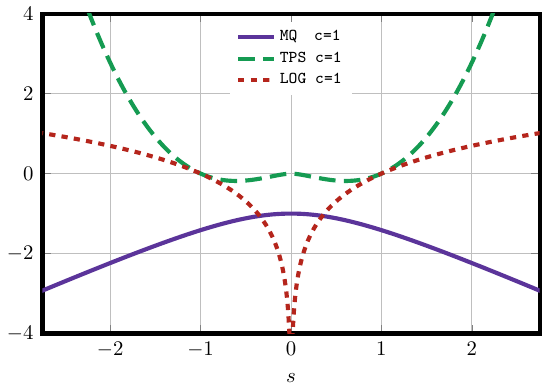}
  \end{subfigure}
  \caption{Graphs of the (kernel) basis functions $F$ in Table \ref{tab:testfunc}.}
  \label{fig:functions}
\end{figure}
\end{figure}

\begin{figure}
  \centering
  \begin{subfigure}{0.48\textwidth}
    \includegraphics[width=\linewidth]{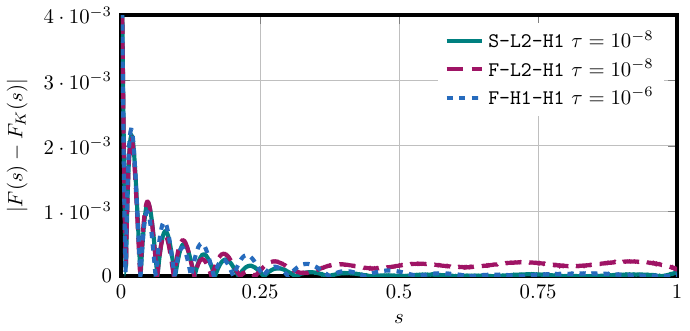}
    \caption{Laplace function with $c=1$.}
    \label{fig:rec_err_Laplace}
  \end{subfigure}
  \hfill
  \begin{subfigure}{0.48\textwidth}
    \includegraphics[width=\linewidth]{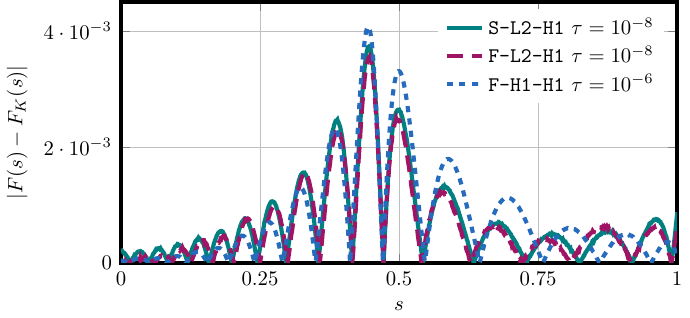}
    \caption{Bump function with $c=\nicefrac12$.}
    \label{fig:rec_err_bump}
  \end{subfigure}

  \caption{Forward error for the methods of Table \ref{tab:methods} in dimension $d=1000$.}
  \label{fig:rec_err}
\end{figure}

\subsection{Forward Error}
First, we investigate the forward error $|F_k(s) - F(s)|$, which should remain small if $f=\mathcal S_d^{-1}[F]$ exists and $J$, $K$ and $L$ are sufficiently large. 
This is to benchmark our discretization scheme.
To have a meaningful error evaluation, we choose a finer discretization than for the coefficient computations.
Specifically, given a (finite) cosine series $f_{\hat a}$, we evaluate $F_K=\mathcal S_d[f_{\hat a}]$ via quadrature as deacribed in Section~\ref{sec:spatial} with twice as many ($2L$) nodes.
For the Laplace and Bump functions, the error $|F_k(s) - F(s)|$ on the uniform grid $\{s_n=\frac{n}{1000}\colon n=0,\ldots, 1000\}$ is plotted in Figure~\ref{fig:rec_err}.
The error is larger in areas where $F(|s|)$ is irregular, i.e., at $s=0$ for Laplace  and at $s=0.4$ for the Bump.
Moreover, it is always below $\num{1e-2}$, which is small enough for our purposes.

\subsection{Fast Kernel Summation Accuracy}

Next, we investigate the fast computation of the kernel sum in \eqref{eq:Fw}.
To this end, we choose i.i.d.\ standard Gaussian samples $x_1,...,x_N\in\R^d$ and $y_1,...,y_M\in\R^d$, as well as i.i.d.\ weights $w_1,\ldots, w_N$ drawn uniformly from $[0,1]$.
The values of $N$, $M$ and $d$ are specified for each experiment separately.
For reference, we compute the sum \eqref{eq:Fw} directly via its definition.
Then, as fast approximation, we use the slicing formula \eqref{eq:SlicingApprox} with $P=d$ random orthogonal directions $\xi_1,\ldots, \xi_P\in \mathbb S^{d-1}$. 
The required coefficients $\hat a$ of the cosine expansion $f_{\hat a}$ are obtained with one of the methods from Table~\ref{tab:methods}.
We assess the approximation quality based on the relative $L^2$-error $\nicefrac{\|s-\hat s\|_2}{\|s\|_2}$, where $s$ denotes the exact evaluation and $\hat s$ the slicing approximation based on $f_{\hat a}$.
All reported $L^2$-errors are averaged over $10$ repetitions of the experiment.
Within this evaluation framework, we investigate the effect of the parameter $\tau$ and the method for computing $\hat a$.

\paragraph{Choice of $\tau$} We have seen in Theorem~\ref{thm:error_dcomp} and Corollary~\ref{cor:to_reg_prob} that the squared $L^2$-error $\|s-\hat s\|_2^2$ decomposes into the forward error $\Vert F_k - F\Vert^2_{\mathcal H}$ and the slicing error \eqref{eq:SliceAndVariation}.
The regularization parameter $\tau$ in \eqref{eq:min_forward_reg} promotes more regular solutions $f_a$.
In principle, these have a smaller slicing error.
Thus, we can see $\tau$ as a mean of balancing the two errors.
To verify this, we compute the relative $L^2$-error for $N=M=\num{1e4}$ samples in dimension $d=1000$, various $\tau$ and all three methods from Table~\ref{tab:methods}.
The results are visualized in Figure~\ref{fig:tau}.
As expected, the error increases for too small and too large values of $\tau$.
Generally, we found that for $d=1000$ the regularization parameters $\tau \in \{10^{-6}, 10^{-7}$, $10^{-4}\}$ work reasonably well for \texttt{S-L2-H1}, \texttt{F-L2-H1} and \texttt{F-H1-H1}, respectively.
Often, the dimension $d$ and the function $F$ (up to a rescaling) remain fixed across multiple evaluations of the kernel sums \eqref{eq:Fw}, allowing $\tau$ to be further tuned for the specific task.

\begin{figure}
  \centering
  \begin{subfigure}{0.48\textwidth}
    \includegraphics[width=\linewidth]{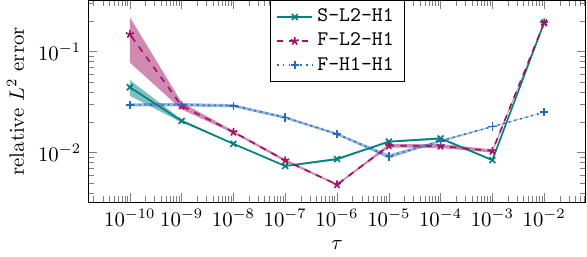}
    \caption{Laplace function with $c=1$.}
    \label{fig:tau_Laplace}
  \end{subfigure}
  \hfill
  \begin{subfigure}{0.48\textwidth}
    \includegraphics[width=\linewidth]{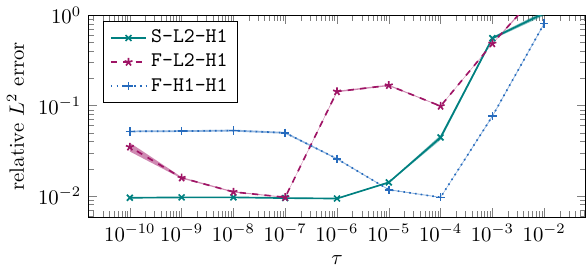}
    \caption{Thin plate spline (TPS) with $c=1$.}
    \label{fig:tau_TPS}
  \end{subfigure}

  \caption{Kernel Summation: Relative $L^2$-error $\|s-\tilde s\|_2/\|s\|_2$ for  $N=M=\num{1e4}$ samples with $P=1000$ slicing directions in dimension $d=1000$.}
  \label{fig:tau}
\end{figure}

\paragraph{Comparison of Methods}
We now compare the quality of the coefficients $\hat a$ obtained with the methods from Table~\ref{tab:methods} for fast kernel summations with the test functions from Table~\ref{tab:testfunc}.
In case that $f = \mathcal S_d^{-1}[F]$ is known, we include slicing via its truncated Fourier series with $K=2^8$ coefficients as baseline (referred to as \texttt{Direct}).
For $N=M=\num{1e4}$ samples in dimension $d \in \{100,1000\}$, the resulting $L^2$-errors are given in Table~\ref{tab:all_aprx_err100} and  \ref{tab:all_aprx_err1000}, respectively.
All three coefficient reconstruction methods yield similar results if $F$ is regular, which indicates that in this case the $L^2$-error is dominated by the slicing error \eqref{eq:SliceAndVariation}.
For the less regular LOG, the Fourier coefficients decay very slowly, such that the \texttt{Direct} method fails. 
In contrast, our method still leads to reasonable estimates.
For Bump and MQ, the underlying $f$ is unknown such that the \texttt{Direct} method cannot be applied.
Again, our methods leads to small relative $L^2$-errors.
In summary, our methods perform at least as good as the \texttt{Direct} method for known $f$.
Moreover, they do not require knowledge of $f$ in order to ensure accurate approximations.

\begin{table}
    \centering
\begin{tabular}{lcccc}
\hline
Function & \texttt{S-L2-H1}  & \texttt{F-L2-H1}  & \texttt{F-H1-H1}  & \texttt{Direct}  \\
Regularization  & $\tau=10^{-6}$ & $\tau=10^{-7}$ & $\tau=10^{-4}$ & - \\
\hline
Gauss $c=1$ & $\num{2.03e-02} $ & $\num{2.03e-02} $ & $\num{2.03e-02} $ & $\num{2.03e-02} $\\
Laplace $c=1$ & $\num{1.93e-02} $ & $\num{4.10e-02} $ & $\num{2.33e-02} $ & $\num{1.83e-02} $\\
IMQ $c=1$  & $\num{7.01e-03} $ & $\num{7.00e-03} $ & $\num{6.99e-03} $ & $\num{7.00e-03} $\\
TPS $c=1$  & $\num{2.83e-02} $ & $\num{2.87e-02} $ & $\num{2.93e-02} $ & $\num{3.12e-02} $\\
LOG $c=1$  & $\num{1.81e-01} $ & $1.58\times 10^{+0}$ & $\num{5.85e-01} $ & $\num{2.99e+01} $\\
MQ $c=1$  & $\num{2.28e-03} $ & $\num{2.33e-03} $ & $\num{2.32e-03} $ & -\\
Bump $c=3$ & $\num{7.51e-03} $ & $\num{1.64e-02} $ & $\num{3.87e-03} $ & -\\

\hline
\end{tabular}
    \caption{Averaged relative $L^2$-error of $s-\tilde s$ over $10$ realizations in dimension $d=100$ with $P=100$ orthogonal slicing directions and $N=M=10^{4}$.
    For all entries, the standard deviation does not exceed $4\%$ of the mean.
}
    \label{tab:all_aprx_err100}
\end{table}

\begin{table}
    \centering
\begin{tabular}{lcccc}
\toprule 
Function & \texttt{S-L2-H1}  & \texttt{F-L2-H1}  & \texttt{F-H1-H1}  & \texttt{Direct}  \\
Regularization  & $\tau=10^{-6}$ & $\tau=10^{-7}$ & $\tau=10^{-4}$ & - \\
\midrule 
Gauss $c=1$ & $\num{6.53e-03} $ & $\num{6.62e-03} $ & $\num{6.61e-03} $ & $\num{6.56e-03} $\\
Laplace $c=1$ & $\num{8.58e-03} $ & $\num{8.32e-03} $ & $\num{1.30e-02} $ & $\num{5.90e-03} $\\
IMQ $c=1$  & $\num{2.25e-03} $ & $\num{2.27e-03} $ & $\num{2.28e-03} $ & $\num{2.26e-03} $\\
TPS $c=1$  & $\num{9.39e-03} $ & $\num{9.64e-03} $ & $\num{9.63e-03} $ & $\num{1.16e-01} $\\
LOG $c=1$  & $\num{1.00e-01} $ & $\num{1.80e-01} $ & $\num{1.55e-01} $ & $\num{2.98e+01} $\\
MQ $c=1$  & $\num{7.30e-04} $ & $\num{7.45e-04} $ & $\num{7.45e-04} $ & -\\
Bump $c=3$ & $\num{3.29e-03} $ & $\num{1.20e-02} $ & $\num{2.69e-03} $ & -\\

\bottomrule
\end{tabular}
    \caption{Averaged relative $L^2$-error of $s-\tilde s$ over $10$ realizations in dimension $d=1000$ with $P=1000$ orthogonal slicing directions and $N=M=10^{4}$.
    For all entries, the standard deviation does not exceed $8\%$ of the mean.
}
    \label{tab:all_aprx_err1000}
\end{table}

\subsection{Runtime Comparison with Brute-Force Method}
So far, we have only investigated the accuracy of our fast kernel summation method based on Algorithm \ref{alg:bessel} and \ref{alg:fourier}.
Next, we want to see how it compares against brute-force GPU implementations in terms of computation time.
To this end, we evaluate the kernel sum \eqref{eq:Fw} in dimension $d=1000$ with a varying number of samples $N=M$.
Here, we expect that the brute-force approach is not competitive anymore for large $N$ and $M$.
All computations are performed with an i7-10700 CPU and a NVIDIA GeForce RTX 2060 GPU.
The reported times are averaged over $10$ runs.

The computation of the cosine expansions $f_a$ based on Algorithms \ref{alg:bessel} and \ref{alg:fourier} consists of several steps, which are independent of the actual summation (slicing overhead).
The respective times of these steps for the MQ function are summarized in Table \ref{tab:runtime}. 
In the first step, one assembles the matrix $H$ or $S$ for Algorithm \ref{alg:bessel} or \ref{alg:fourier}, respectively (column ``1.\ Mat'').
Since $H$ and $S$ depend only on $d$ and the hyperparameters $J$, $K$ and $L$, we may store them for future use.
Then, this step does not add towards the overhead anymore.
Secondly, the right-hand side vector $b$ is assembled (``2.\ Rhs'').
For Algorithm \ref{alg:fourier}, we use the FFT with an oversampling factor of $4$.
This step is computationally negligible.
In the third step, we solve a least squares problem using \texttt{torch.linalg.lstsq} (column ``3.\ Solve'').
For both algorithms, the total computations times to obtain the cosine coefficients are roughly the same.
Finally, the columns ``Fastsum'' and ``PyKeOps'' in Table \ref{tab:runtime} contain the time for actually evaluating \eqref{eq:Fw} using the fast summation and the brute-force computation with the highly optimized package PyKeOps \cite{CFGCD2011}, respectively.
The runtime of PyKeOps depends on $F$.
While the Gauss, Laplace, IMQ, MQ and LOG function take roughly the same amount of time, the Bump and TPS take significantly longer.
In contrast, the runtime for the fast summation combined with Algorithm \ref{alg:bessel} depends only marginally on $F$.
For $N=M=10^4$ samples, the sum evaluation takes longer than the computation of the cosine expansion $f_a$.

Next, we have a closer look at the empirical behavior with respect to $N$ and $M$ for the MQ function with $c=1$.
While the brute-force evaluation scales as $\mathcal{O}(MNd)$, the fast summation has the improved complexity $\mathcal{O}(dP(N+M+K\log K))$.
In Figure~\ref{fig:runtime}, we show the evaluation times on the GPU for different $N=M$ with the \texttt{S-L2-H1} method (excluding compilation times).
We prebuilt $H$ so that its construction does not add towards the reported times.
As expected, the brute-force approach exhibits quadratic scaling with $N=M$.
As soon as $N\ge \num{3e3}$, the fast summation method is faster even when we include the overhead from Algorithm \ref{alg:bessel}.
For $N = \num{5e4}$ samples, the fast summation yields a speedup of roughly a factor of 50.
In many applications, both the kernel and the dimension $d$ are known a priori, allowing $f_a$ to be precomputed and stored for later use.
Then, the slicing overhead vanishes, making the approach relevant for even smaller sample sets.
A more detailed comparison of brute-force computation and fast summations based on slicing can be found in \cite{HJQ2025}.

\begin{table}
    \centering

\begin{tabular}{ccc|cccc}
\toprule
Alg. & Dev. &  1.\ Mat  & 2.\ Rhs  & 3.\ Solve & Fastsum & PyKeOps \\
\midrule 
0 & cuda  & $\num{7.84e-02} $ & $\num{2.33e-03} $ & $\num{4.27e-03} $ & $\num{9.46e-02} $ & $\num{4.15e-01} $\\
1 & cuda  & $\num{1.02e-01} $ & $\num{2.30e-03} $ & $\num{4.24e-03} $ & $\num{9.49e-02} $ & $\num{4.15e-01} $\\

0 & cpu  & $\num{4.94e-01} $ & $\num{6.44e-05} $ & $\num{1.95e-03} $ & $\num{6.47e-01} $ & $\num{1.83e+01} $\\
1 & cpu  & $\num{1.31e+00} $ & $\num{3.11e-04} $ & $\num{1.90e-03} $ & $\num{6.61e-01} $ & $\num{1.92e+01} $\\


\bottomrule
\end{tabular}

    \caption{Runtime in seconds with $L=2^{10}$ and $K=J=2^8$ averaged over $10$ runs for MQ with $c=1$ as reference.
    The last two columns are for $N=M=10^4$ samples.
    For all entries, the standard deviation does not exceed $15\%$ of the mean.}  \label{tab:runtime}
\end{table}

\begin{figure}
  \centering
  \begin{subfigure}{0.48\textwidth}
    \includegraphics[width=\linewidth]{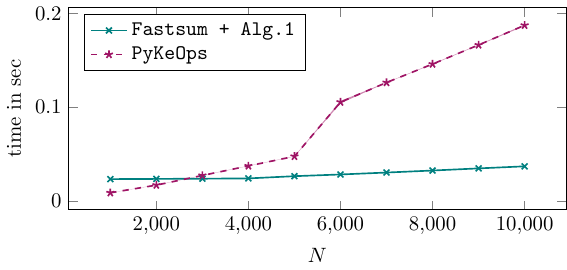}
    \caption{For data size $N\in \{ 10^3n\colon n=1\ldots, 10\}$. }
    \label{fig:runtime_10000}
  \end{subfigure}
  \hfill
  \begin{subfigure}{0.47\textwidth}
    \includegraphics[width=\linewidth]{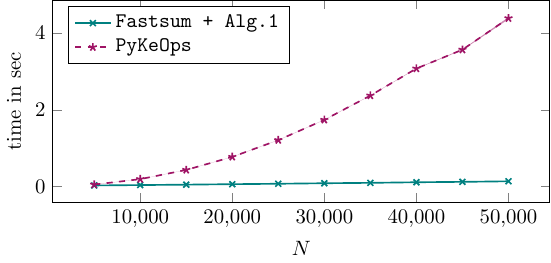}
    \caption{For data size $N\in \{ 5\cdot 10^3n\colon n=1\ldots, 10\}$.}
    \label{fig:runtime_50000}
  \end{subfigure}

  \caption{Runtime on GPU for brute-force computation with PyKeOps and fast summation using slicing for various sample numbers $N=M$ and MQ as reference.}
  \label{fig:runtime}
\end{figure}


\section{Conclusions}\label{sec:concl}

Fast Fourier summation techniques have the potential to massively accelerate kernel summations.
To make them practical, we proposed a method to numerically approximate the  required slicing function $f$ of the kernel.
For various kernels and dimensions (up to $d=1000$), the approximation of $f$ remained computationally inexpensive.
Moreover, the efficiency of our method is theoretically supported by error bounds in $L^2$ and $H^1$.
When deploying the obtained $f$ for fast summations, the computational break even point compared to the brute-force approach already occurs for around $N=3000$ samples.
Thus, exploring the practical benefits of the proposed acceleration techniques in applications seems a promising direction of future work.

\section*{Acknowledgment}

JH acknowledges funding from the German Research Foundation (DFG) within the Walter Benjamin Programme under the project number 530824055.
NR acknowledges funding from the European Union and the Free State of Saxony (ESF).
We thank Michael Quellmalz for his valuable suggestions, especially regarding numerical aspects of slicing.

\bibliographystyle{abbrv}
\bibliography{ref}

\appendix
\section{Additional Proofs}
Let $\mathcal S(\R^d)$ denote the Schwartz space  of all rapidly decreasing functions in $C^\infty(\R^d)$ and $\mathcal S'(\R^d)$ the corresponding dual space of tempered distributions.
The corresponding radial spaces are denoted by $\mathcal S_\mathrm{rad}(\R^d)$ and $\mathcal S'_\mathrm{rad}(\R^d)$.
Now, we recall some operators that were defined in \cite{RQS2025}.
Let $\mathcal R_d\colon\mathcal S_\mathrm{rad}(\R)\to\mathcal S_\mathrm{rad}(\R^d)$ be defined as $\mathcal R_d[F]=F\circ\|\cdot\|$ and $\mathcal A_d\colon \mathcal S(\R^d)\to\mathcal S_\mathrm{rad}(\R^d)$ be its left-inverse given by
$\mathcal A_d\Phi(r)=\omega_{d-1}^{-1}\int_{\mathbb S^{d-1}} \Phi(|r|\xi)\d \xi$.
Moreover, let $\mathcal M_d\colon\mathcal S_\mathrm{rad}(\R)\to \mathcal S_\mathrm{rad}(\R)$ be defined by $\mathcal M_d[f](x)\coloneqq f(|x|)|x|^{d-1}$.
We extend these operators to $\mathcal S'_\mathrm{rad}(\R^d)$ via duality as 
$\langle \mathcal A_d T, \psi\rangle=\langle  T, \mathcal A_d\psi\rangle$ and $\langle \mathcal M_d T, \psi\rangle=\langle  T, \mathcal M_d\psi\rangle$.
Additionally, we consider the ``adjoint'' $\mathcal R_d^\star\colon \mathcal S'(\R^d)\to \mathcal S_\mathrm{rad}'(\R)$ defined by $\langle \mathcal R_d^\star T,\psi\rangle = \langle T, \mathcal R_d\circ \mathcal A_1\psi\rangle$.
Then, it holds $f = (\mathcal F_1\circ \mathcal R_d^\star \circ \mathcal F_d^{-1})[\mathcal R_d F]$ in the sense of distributions \cite[Thm.~4.7]{RQS2025}.
\subsection{Proof of Proposition~\ref{prop:imq}}\label{proof:imq}
By \cite[Thm.~8.15]{Wendland2004}\footnote{We chose $\beta=-\frac{1}{2}$ and converted the definition \cite[Def.~5.15]{Wendland2004} of the Fourier transform to ours \eqref{eq:def_fourier}.}, $\Phi(x)\coloneqq F(\|x\|) = (c^2+\|x\|^2)^{-\nicefrac{1}{2}}$ is positive definite and has a generalized Fourier transformation on $\R^d$ of order $m=0$ given by
\begin{equation}
\hat \Phi (\omega)
= 2\bigl( \tfrac{c}{\|\omega\|}\bigr)^{\nicefrac{(d-1)}{2}} K_{\nicefrac{(d-1)}{2}} (2\pi c\|\omega\|)d,
\end{equation}
where $K_\nu$ denotes the modified Bessel function of the second kind.
By \cite[Prop.~4.5 i)]{RQS2025} it holds $\langle \mathcal R^\star \hat \Phi, \psi\rangle  = \tfrac{\omega_{d-1}}{2} \langle (\mathcal M_d \circ \mathcal A_d)\hat \Phi, \psi \rangle$ for any $\psi \in \mathcal S(\R)$.
Therefore, we have
\begin{equation}
    (\mathcal{R}^\star \hat \Phi)(r)
    =\omega_{d-1} (c|r|)^{\nicefrac{(d-1)}{2}} K_{\nicefrac{(d-1)}{2}}(2\pi c |r|).
\end{equation}
Now, applying \cite[Thm.~8.15]{Wendland2004} with $\tilde \beta = -\nicefrac{d}{2}$, $\tilde d=1$ and $f(r)=c^{d-1}(c^2+r^2)^{\tilde \beta}$, we obtain using $K_\nu=K_{-\nu}$ that
\begin{align}
\hat f(r)
& =\frac{ c^{d-1}(2\pi)^{\nicefrac{1}{2}}  2^{1-\nicefrac{d}{2}}}{\Gamma(\nicefrac{d}{2})} \Bigl(\frac{c}{2\pi |r|}\Bigr)^{-\nicefrac{(d-1)}{2}} K_{-\nicefrac{(d-1)}{2}}(2\pi c |r|)\notag\\
&=\frac{\sqrt{2\pi} 2^{-\nicefrac{(d-2)}{2}} (2\pi)^{\nicefrac{(d-1)}{2}} }{\Gamma(\nicefrac{d}{2}) } (c|r|)^{\nicefrac{(d-1)}{2}}K_{\nicefrac{(d-1)}{2}}(2\pi c|r|)\notag\\
&= \omega_{d-1} (c|r|)^{\nicefrac{(d-1)}{2}}K_{\nicefrac{(d-1)}{2}}(2\pi c|r|)
=  \mathcal R_d^\star \hat \Phi (r).
\end{align}
Since the Fourier transform for even distributions satisfies $\mathcal F_d=\mathcal F_d^{-1}$, we have $\hat f = \mathcal R_d^\star \hat \Phi$ and 
$ f = (\mathcal F_1\circ \mathcal R_d^\star \circ \mathcal S_d^{-1})[\mathcal R_d F]$.
Finally, since $f\in \mathcal C(\R)$ is slowly increasing, \cite[Thm.~4.7]{RQS2025} yields $\mathcal S_d[f]=F$.

\subsection{Proof of Proposition~\ref{prop:Sd_inv_sum_exp}}\label{sec:ProofProp4}
\paragraph{Part (i)}
First, we show by induction over $m$ that with $\alpha=d-2$ it holds
\begin{equation}\label{eq:help_indct}
    \frac{\d^m}{\d t^m}\bigl(F(\sqrt{t})\sqrt{t}^\alpha\bigr)=2^{-m}\sum_{k=0}^m a_{m,k} F^{(k)}(\sqrt{t})\sqrt{t}^{\alpha-2m+k}
\end{equation}
for $m\le \nicefrac{(d-1)}{2}$.
The base case $m=0$ follows directly.
For general $\tilde \alpha \neq 0$, we have
\begin{align}\label{eq:ind_step}
&\frac{\d}{\d t}\bigl(F(\sqrt{t})\sqrt{t}^{\tilde \alpha}\bigr)
=\frac{1}{2}F'(\sqrt{t})\sqrt{t}^{\tilde\alpha-1} + \frac{\tilde \alpha}{2} F(\sqrt{t})\sqrt{t}^{\tilde \alpha-2}.
\end{align}
Now, assume that \eqref{eq:help_indct} is true for some $m$ with $m+1\le \nicefrac{(d-1)}{2}$.
Then $\alpha-2m+k \ge 1$ for $k=0,\ldots, m$ so that we can apply the induction assumption and \eqref{eq:ind_step} to obtain
\begin{align}
    &\frac{\d^{m+1}}{\d t^{m+1}}(F(\sqrt{t})\sqrt{t}^\alpha)
    =2^{-m} \sum_{k=0}^m a_{m,k} \frac{\d}{\d t}\bigl(F^{(k)}(\sqrt{t})\sqrt{t}^{\alpha-2m+k}\bigr)\notag\\
    =&2^{-m} \sum_{k=0}^m \frac{a_{m,k}}{2}\left((\alpha-2m+k)F^{(k)}(\sqrt{t})\sqrt{t}^{\alpha-2m+k-2}+F^{(k+1)}(\sqrt{t})\sqrt{t}^{(\alpha-2m+k-1)}\right)\notag\\
    =&2^{-(m+1)} \sum_{k=0}^{m+1} a_{m+1,k} F^{(k)}(\sqrt{t})\sqrt{t}^{\alpha-2(m+1)+k}.
\end{align}
This concludes the induction.
Inserting \eqref{eq:help_indct} with $m=\nicefrac{(d-1)}{2}$ into the inversion formula \eqref{eq:RLFI_inv}, which is restated in \eqref{eq:InverseSeries} for convenience, gives with $n=\nicefrac{(d-1)}{2}$ that
\begin{equation}\label{eq:InverseSeries}
\mathcal S_d^{-1}[F](t) = \frac{2t}{c_d\Gamma(n)} D^n_+\bigl[F(\sqrt{s})\sqrt{s}^{d-2}\bigr] (t^2)
=\frac{1}{2^{n-1}c_d\Gamma(n)}\sum_{k=0}^n a_{n,k} F^{(k)}(t)t^{k}.
\end{equation}
The leading coefficient in \eqref{eq:InverseSeries} can be further simplified with the Legendre duplicate formula to $2^{n-1}c_d\Gamma(n) = \nicefrac{(2n)!}{(2^nn!)}$.

\paragraph{Part (ii)}
For $F\in \mathcal C^\infty([0,1])$, it holds by part i) that
\begin{align}\label{eq:Sd_inv_bound}
\|\mathcal S_d^{-1}F\|_{L^2([0,1])}
\le \left(\tfrac{2^nn!}{(2n)!}(a_{n,0}+\ldots,+a_{n,n})\right)^{\nicefrac12}\|F\|_{H^n([0,1])}.
\end{align}
Now, the density of $C^\infty([0,1])$ in $H^n([0,1])$ (Mayers-Serrin theorem) and the BLT theorem imply that we can extend $\mathcal S_d^{-1}\colon C^\infty([0,1]) \to L^2([0,1])$ to a bounded operator $\mathcal S_d^{-1}  \colon H^n([0,1]) \to L^2([0,1])$ with norm $C_d \coloneqq (\tfrac{2^nn!}{(2n)!}(a_{n,0}+\ldots,+a_{n,n}))^{\nicefrac12}$.
Since the expression \eqref{eq:poly_expan} is bounded, it coincides with $\mathcal S_d$ also on $H^n([0,1])$.

\subsection{Proof of Theorem~\ref{thm:S_d_embeds}}\label{proof:S_d_embeds}
We use the \emph{Hardy operator} $\mathcal I[f](s)\coloneqq s^{-1}\int_0^s f(t)\d t$ and the \emph{Hardy inequality} \cite{H1920,PS2015}, which states that $\mathcal  I\colon L^p([0,1])\to L^p([0,1])$ is bounded with $\|\mathcal I\|_\mathrm{op}= \frac{p}{p-1}$.
Now, since $L^p([0,1])\subseteq L^1([0,1])$, the expression $\mathcal S_d[f]$ is well-defined and
\begin{equation}
|\mathcal S_d[f](s)|
\le \int_0^1 |f(ts)|\varrho_d(t)\d t
\le c_d \int_0^1 |f(ts)|\d t
= c_ds^{-1} \int_0^s |f|(t)\d t
=c_d \mathcal I[|f|](s).
\end{equation}
Hence, $\mathcal S_d\colon L^p([0,1])\to L^p([0,1])$ is bounded with
\begin{equation}
    \|\mathcal S_d[f]\|_{L^p([0,1])}
    \le c_d\bigl\|\mathcal I[|f|]\bigr\|_{L^p([0,1])}
    \le c_d \tfrac{p}{p-1} \|f\|_{L^p([0,1])}.
\end{equation}
The case $p=\infty$ follows from Hölder's inequality
\begin{equation}\label{eq:BoundInfty}
|\mathcal S_d[f](s)|\le \int_0^1 |f(ts)|\varrho_d(t)\d t
\le \|f\|_{L^\infty([0,1])} \|\varrho_d\|_{L^1([0,1])}
= \|f\|_{L^\infty([0,1])}.
\end{equation}
Note that the bound \eqref{eq:BoundInfty} is sharp for $f\equiv 1$.
For $p=1$ and $w\in L^1([0,1],\nicefrac{1}{s})$, we get
\begin{align}
\|\mathcal S_d[f]\|_{L^1([0,1],w)}
&
=\int_0^1 |\mathcal S_d[f](s)|w(s)\d s
\le \int_0^1 \frac{w(s)}{s}c_d\int_0^s |f(t)|\d t \d s\notag\\&
=c_d\int_0^1 |f(t)|\int_t^1 \frac{w(s)}{s}\d s \d t
=c_d\|f\|_{L^1([0,1])} \|w\|_{L^1([0,1],\nicefrac{1}{s})}.
\end{align}
This concludes the proof for $f \in L^p([0,1])$.

For $f\in H^1([0,1])$, the Leibniz formula and Jensen's inequality imply that
\begin{equation}
\|\mathcal S_d[f]'\|_{L^2([0,1])}^2
=\int_0^1\biggl| \int_0^1 tf'(ts)\varrho_d(t)\d t\biggr|^2 \d s
\le \int_0^1 \int_0^1 |tf'(ts)|^2\varrho_d (t)\d t \d s.
\end{equation}
Then, $\varrho_d(t) \leq c_d$, a substitution, and Tonelli's theorem yield
\begin{align}
&\|\mathcal S_d[f]'\|_{L^2([0,1])}^2
\le c_d \int_0^1 \frac{1}{s^3}\int_0^s |tf'(t)|^2  \d t \d s
=c_d \int_0^1 |tf'(t)|^2 \int_t^1\frac{1}{s^3}\d s\d t\notag\\
\quad & =c_d \int_0^1 |tf'(t)|^2 \left[ \frac{-1}{2s^2}\right]_t^1\d t
=c_d \int_0^1 |tf'(t)|^2 \frac{1 -t^2}{2t^2}\d t
=\frac{c_d}{2}\int_0^1 |f'(t)|^2 (1-t^2)\d t.
\end{align}
Since $(1-t^2) \leq 1$ for $t\in [0,1]$, we further get that
\begin{equation}\label{eq:BoundFirstDer}
\|\mathcal S_d[f]'\|_{L^2([0,1])}^2
\le \frac{c_d}{2} \int_0^1 |f'(t)|^2\d t 
= \frac{c_d}{2} \|f'\|_{L^2([0,1])}^2
\le \frac{c_d}{2} \|f\|_{H^1([0,1])}^2.
\end{equation}
To estimate $\|\mathcal S_d[f]\|_{L^2([0,1])}$, we use the Sobolev inequality and \eqref{eq:BoundInfty} to get
\begin{equation}\label{eq:BoundLP}
\|\mathcal S_d [f]\|_{L^2([0,1])}^2
\le \|\mathcal S_d [f]\|_{L^\infty([0,1])}^2
\le \|f\|_{L^\infty ([0,1])}^2
\le 2\|f\|_{H^1([0,1])}^2.
\end{equation}
By combining \eqref{eq:BoundFirstDer} and \eqref{eq:BoundLP}, we can now get $\|\mathcal S_d\|_\mathrm{op}\le(2+\frac{c_d}{2})^{\nicefrac{1}{2}}$.

\subsection{Proof of Proposition~\ref{prop:g}}\label{proof:g}

We first prove an auxiliary lemma.

\begin{lemma}\label{lem:deriv_of_Sd}
If $f\in \mathcal C^2([0,1])$ and $d\ge 3$, then it holds that $|\mathcal S_d[f]'|_\infty \le \tfrac{c_{d+2}}{d}\|f'\|_\infty$.
If additionally $f'(0) = 0$, then $\mathcal S_d[f]'(s)=\tfrac{s}{d}\mathcal S_{d+2}[f''](s)$.
\end{lemma}
\begin{proof}
Since $f\in \mathcal C^2([0,1])$ and $\varrho_d$ is bounded, we obtain by the Leibnitz integral formula
\begin{equation}\label{eq:derivative_Sd}
\mathcal S_d[f]'(s)
=\frac{\d}{\d s} \int_0^1 f(ts)\varrho_d(t)\d t
=\int_0^1 \frac{\d}{\d s} f(ts)\varrho_d(t)\d t
=\int_0^1 tf'(ts)\varrho_d(t)\d t.
\end{equation}
Further, it holds for $t\in (-1,1)$ that
\begin{align}\label{eq:ReformulationVarrho}
\varrho_{d+2}'(t)
&
= - 2c_{d+2} \frac{(d-1)}{2} t (1-t^2)^{\nicefrac{(d-3)}{2}}
= - \frac{2\Gamma(\frac{d}{2}+1)}{\sqrt{\pi}\Gamma(\frac{d-1}{2}+1)} (d-1) t (1-t^2)^{\nicefrac{(d-3)}{2}}
\notag\\&
= - \frac{2\Gamma(\frac{d}{2})\frac{d}{2}}{\sqrt{\pi}\Gamma(\frac{d-1}{2})\frac{d-1}{2}} (d-1)t (1-t^2)^{\nicefrac{(d-3)}{2}}
=-dt \varrho_d(t).
\end{align}
After inserting \eqref{eq:ReformulationVarrho} into \eqref{eq:derivative_Sd}, the estimate follows from Hölder's inequality
\begin{equation}
|\mathcal S_d[f]'(s)|
= \left| -\frac{1}{d}\int_0^1 f'(ts)\varrho_{d+2}'(t)\d t\right|
\le \frac{1}{d} \|f'\|_\infty [-\varrho_{d+2}]_0^1
=\frac{c_{d+2}}{d}\|f'\|_\infty.
\end{equation}
Similarly, by applying partial integration instead, we obtain using $f'(0) = 0$ that
\begin{equation}
\mathcal S_d[f]'(s)
=-\frac{1}{d}\int_0^1 f'(ts)\varrho_{d+2}'(t)\d t
=\frac{1}{d}\int_0^1 s f''(ts)\varrho_{d+2}(t)\d t
=\frac{s}{d}\mathcal S_{d+2}[f''](s).
\end{equation}
\end{proof}

For the first item of Proposition~\ref{prop:g}, note that
\begin{equation}\label{eq:h_k_01}
\eta_{d}(s)=\mathcal S_d[\cos](s)
=\frac12 \int_{-1}^1 \cos(ts) \varrho_d(t)\d t
=\frac{c_d}{2} \int_{-1}^1 (1-t^2)^{\frac{d-3}{2}} \exp(\i s t) \d t.
\end{equation}
As shown in \cite[Eq.~3.387 2.]{GR2015}, it holds for any $\nu>0$ that
\begin{equation}\label{eq:SpecialIntegral}
\int_{-1}^1 (1-t^2)^{\nu-1}\exp({\i s t})\d t = \sqrt{\pi} \Bigl(\frac{2}{s}\Bigr)^{\nu-\nicefrac{1}{2}} \Gamma(\nu)J_{\nu-\nicefrac{1}{2}}(s).
\end{equation}

By inserting \eqref{eq:SpecialIntegral} with $\nu =\nicefrac{(d-1)}{2}$ into \eqref{eq:h_k_01}, we get that
\begin{align}\label{eq:Represent_hd}
\eta_d(s)
= \frac{c_d}{2}\sqrt{\pi} \Bigl(\frac{2}{s}\Bigr)^{\nicefrac{(d-2)}{2}} \Gamma\Bigl(\frac{d-1}{2}\Bigr)J_{\nicefrac{(d-2)}{2}}(s)
=\Gamma\Bigl(\frac{d}{2}\Bigr)\Bigl(\frac{2}{s}\Bigr)^{\nicefrac{(d-2)}{2}}J_{\nicefrac{(d-2)}{2}}(s).
\end{align}
This is exactly the confluent hypergeometric function ${_0}F_1(\nicefrac{d}{2}, -\nicefrac{s^2}{4})$, see \eqref{eq:bessel_0F1}.
In \cite[Eq.~8.402]{GR2015}, the power series expansion of $J_\nu$ is given as
\begin{equation}\label{eq:BesselExpansion}
    J_\nu(z)=\frac{z^\nu}{2^\nu}\sum_{k=0}^\infty (-1)^k \frac{z^{2k}}{2^{2k} k!\Gamma(\nu+k+1)}\quad  z\ge 0.
\end{equation}
Plugging \eqref{eq:BesselExpansion} into \eqref{eq:Represent_hd}, it follows that
\begin{equation}
   \eta_d(s)
    =\Gamma\Bigl(\frac{d}{2}\Bigr)\sum_{k=0}^\infty \frac{(-1)^k}{\Gamma(k+1)\Gamma(\nicefrac{d}{2}+k)}\Bigl(\frac{s}{2}\Bigr)^{2k}.
\end{equation}
The second item follows from Hölder's inequality 
\begin{align}
    |\eta_d(s)|=\frac{1}{2}\biggl|\int_{-1}^1 \cos(ts)\varrho_d(t)\d t\biggr|
    \le \|\mathcal D_s\cos\|_{L^\infty([-1,1])} \frac{\|\varrho_d\|_{L^1([-1,1])}}{2}
    =1.
\end{align}
The last item follows from Lemma~\ref{lem:deriv_of_Sd} by inserting $f=\cos$.

\section{Least Squares Solutions and Truncation}\label{app:not_truncated}

Even when $f=\mathcal S_d^{-1}[F]$ exists, the solution of \begin{equation}\label{eq:min_forward2}
    \hat a=\argmin_{a\in \R^{K}} \Vert \mathcal S_d[f_a]-F\Vert_{\mathcal H}^2 = \argmin_{a\in \R^{K}} \Vert  a_0h_{0}+\ldots+a_Kh_{K-1}-F\Vert_{\mathcal H}^2
\end{equation}
does not necessarily coincide with the truncated cosine series of $f$.
We provide a counterexample.
Let $\mathcal H=L^2([0,1])$, $d \geq  7$ and $K=\lfloor \nicefrac{d}{(2\pi)}\rfloor $.
Define $f(t)=g_{K}(t)=\sqrt{2}\cos(\pi k t)$ and $F=\mathcal S_d[f]$.
Then $a_k=\langle g_k,f\rangle=\langle g_k,g_{K}\rangle=0$ for all $k=0,...,K-1$.
Now, we prove that the solution of \eqref{eq:min_forward2} is not zero.
To this end, we note that the optimality condition for \eqref{eq:min_forward2} reads $Ha=b$, where $H_{j,k}=\langle h_j,h_k\rangle_{\mathcal H}$ and $b_k=\langle F,h_k\rangle_{\mathcal H}$.
In Lemmas \ref{lem:Invertible} and \ref{lem:PositiveEntries}, we show that $H\in\R^{K\times K}$ is invertible and that $b\neq 0$.
Then, the minimizer $\hat a\neq 0$ cannot coincide with the truncated cosine transform of $f$.

\begin{lemma}\label{lem:Invertible}
    Let $H_{j,k}=\langle h_j,h_k\rangle_{\mathcal H}$ for $h_k=\mathcal S_d[g_k]$ and $j,k=0,...,K-1$. Then, the matrix $H\in\R^{K\times K}$ is symmetric positive definite and thus invertible.
\end{lemma}
\begin{proof}
Since $\{g_{k}\}_{k=0}^{K-1}$ are linearly independent and $\mathcal S_d$ is injective on $L^2([0,1])$ by \cite[eq.~(8)]{L1971}, it holds for every $a\in \R^{K}\setminus \{0\}$ that
\begin{equation}\label{eq:InnerProdH}
    a^\top H a
    =\sum_{j,k=0}^{K-1} a_ka_j \langle h_{k}, h_{j}\rangle_{\mathcal H}
    = \|  \mathcal S_d [a_0g_{0}+\ldots a_Kg_{K-1}]\|_{\mathcal H}^2> 0.
\end{equation}
\end{proof}

\begin{lemma}\label{lem:PositiveEntries}
Let $d\ge 3$ and $0\le k,j\le \nicefrac{d}{(2\pi)}$, then $h_{k}$ is positive and $\langle h_{k},h_{j}\rangle_{L^2([0,1])}>0$.
\end{lemma}
\begin{proof}
For any $-1<\nu<\infty$, the first zero $j_{\nu,1}>0$ of the Bessel function $J_\nu$ satisfies $j_{\nu,1}^2>(\nu+1)(\nu+5)$ \cite{L1993}.
Since $\eta_d(s)=\Gamma(\nicefrac{d}{2})(\nicefrac{2}{s})^{\nicefrac{d}{2}-1}J_{\nicefrac{d}{2}-1}(s)$, we can thus lower bound the first zero of $\eta_d$ by $\nicefrac{d}{2}$.
Consequently, we know that for $k\ge 1$ the function $h_{k}=\sqrt{2} \mathcal D_{\pi k} \eta_d$ is positive on $[0,\nicefrac{d}{(2\pi k)})$.
Hence, $h_{k}$ is positive for $1\le k\le \nicefrac{d}{(2\pi)}$, and it holds that $\langle h_{k},h_{j}\rangle_{L^2([0,1])} >0$ for all $1\le k,j\le\nicefrac{d}{(2\pi)}$.
\end{proof}
\end{document}